\documentclass[a4paper,12pt]{amsart}

\usepackage{amssymb}
\usepackage{amsmath}
\usepackage{amsthm}
\usepackage{mathtools}
\usepackage{enumitem}
\usepackage{graphicx}
\usepackage{tikz}
\usepackage{mathrsfs}
\usepackage[mathscr]{euscript}

\usetikzlibrary{cd}
\graphicspath{ {.} }

\usepackage[english]{babel}

\newcommand{\innerthmname}{}

\theoremstyle{definition}

\usetikzlibrary{positioning}
\newtheorem{theorem}[equation]{Theorem}
\newtheorem{lemma}[equation]{Lemma}

\newtheorem{corollary}[equation]{Corollary}

\theoremstyle{definition}
\newtheorem{definition}[equation]{Definition}

\theoremstyle{remark}
\newtheorem{example}[equation]{Example}
\newtheorem{remark}[equation]{Remark}

\numberwithin{equation}{section}
\usepackage{fancyhdr}
\usepackage{amsfonts}
\usepackage{hyperref}
\usepackage{graphicx, xcolor} 
\usepackage{amsthm}

\usepackage[
margin=1in,
marginpar=2cm,
includefoot,
footskip=30pt,
]{geometry}

\usepackage{scalerel,stackengine}
\stackMath
\newcommand\reallywidehat[1]{%
	\savestack{\tmpbox}{\stretchto{%
			\scaleto{%
				\scalerel*[\widthof{\ensuremath{#1}}]{\kern-.6pt\bigwedge\kern-.6pt}%
				{\rule[-\textheight/2]{1ex}{\textheight}}
			}{\textheight}%
		}{0.5ex}}%
	\stackon[1pt]{#1}{\tmpbox}%
}

\title[Relative Generalized Boolean Dynamical System Algebras]{Relative Generalized Boolean Dynamical System Algebras}
\author[A. Zhang]{Allen Zhang}
\email{allenusca@gmail.com}

\begin{document}
	
	\begin{abstract}
		We study an algebraic analog of a $C^{\ast}$-algebra associated to a generalized Boolean dynamical system which parallels the relation between graph $C^{\ast}$-algebras and Leavitt path algebras. We prove that such algebras are Cuntz-Pimsner algebras and partial skew group rings and use these facts to prove a graded uniqueness theorem. We then describe the notion of a relative generalized Boolean dynamical system and generalize the graded uniqueness theorem to relative generalized Boolean dynamical system algebras. We use the graded uniqueness theorem to characterize the graded ideals of a relative generalized Boolean dynamical system algebra in terms of the underlying dynamical system. We prove that every generalized Boolean dynamical system algebra is Morita equivalent to an algebra associated to a generalized Boolean dynamical system with no singularities. Finally, we give an alternate characterization of the class of generalized Boolean dynamical system algebras that uses an underlying graph structure derived from Stone duality.
	\end{abstract}

	\maketitle

	\section{Introduction}
	The study of combinatorial objects and their associated $C^{\ast}$-algebras was started by Cuntz and Krieger in \cite{Cuntz1980} which gave way to a class of $C^{\ast}$-algebras known as graph $C^{\ast}$-algebras. As research on these $C^{\ast}$-algebras has continued, one avenue has been to generalize the combinatorial objects which $C^{\ast}$-algebras are generated from. One generalization of the graph $C^{\ast}$-algebras was the labelled graph $C^{\ast}$-algebras introduced by Bates and Pask in \cite{bates2007c}. A further generalization was the generalized Boolean dynamical system $C^{\ast}$-algebra, which was developed by De Castro, Carlsen, and Kang in \cite{CARLSEN2020124037, Castro2022CalgebrasOG, DECASTRO2023126662}.

	In the study of combinatorial $C^{\ast}$-algebras, an algebraic analog was found to preserve many interesting properties. This has bloomed into its own field of study, with the most well-studied of these analogs being the Leavitt path algebra which is the algebraic analog of a graph $C^{\ast}$-algebra. See \cite{AbramsSurvey} by Abrams for a historical exposition of the Leavitt path algebra. Recently, an algebraic analog of the labelled graph $C^{\ast}$-algebra was studied by Boava, de Castro, Gonçalves, and van Wyk in \cite{Boava2021LeavittPA} which they termed a labelled Leavitt path algebra. In this paper, we continue this line of work by constructing the algebraic analog of a generalized Boolean dynamical system $C^{\ast}$-algebra, which we call a generalized Boolean dynamical system algebra.

	In Section~\ref{prelim}, we review the definition of a (relative) generalized Boolean dynamical system and associate to a unital commutative ring $R$ and relative generalized Boolean dynamical system an $R$-algebra. In Section~\ref{partialskewgroup}, we use a technique from \cite{zhang2025partialactionsgeneralizedboolean} to realize a generalized Boolean dynamical system algebra as a partial skew group ring arising from an inverse semigroup associated to a relative generalized Boolean dynamical system in \cite{DECASTRO2023126662}. In Section~\ref{relativetonon}, we use a construction from \cite{CARLSEN2020124037} to realize relative generalized Boolean dynamical system algebras as generalized Boolean dynamical system algebras. 
	
	In Section~\ref{section:gradedunique}, we follow a technique used in \cite{Boava2021LeavittPA} for labelled Leavitt path algebras to realize generalized Boolean dynamical system algebras as Cuntz-Pimsner algebras and prove a graded uniqueness theorem. We then apply theorems from Section~\ref{relativetonon} to derive a graded uniqueness theorem for the more general relative generalized Boolean dynamical system algebras. In Section~\ref{section:gradedideal}, we mention a characterization for the graded ideals of a relative generalized Boolean dynamical system algebra that follows from the graded uniqueness theorem and ideas in \cite{CARLSEN2020124037} used to characterize the gauge-invariant ideals of a relative generalized Boolean dynamical system $C^{\ast}$-algebra.

	In Section~\ref{section:desing}, we prove a desingularization result for generalized Boolean dynamical system algebras by leveraging the associated inverse semigroup from Section~\ref{partialskewgroup} and applying theorems from \cite{zhang2025moritaequivalencesubringsapplications}. Desingularization results are known for both Leavitt path algebras and graph $C^{\ast}$-algebras \cite{abrams2008leavitt, Drinen2000}. A desingularization result has been established for weakly left-resolving labelled space $C^{\ast}$-algebras by Banjade, Chambers, and Ephrem in \cite{banjade2024singularities} and for labelled Leavitt path algebras of weakly left-resolving normal labelled spaces by Zhang in \cite{zhang2025moritaequivalencesubringsapplications}. To best of the author's knowledge, this is the first desingularization construction for a generalized Boolean dynamical system in either the $C^{\ast}$-algebra or algebraic case.
	
	Finally, in Section~\ref{section:labelled}, we define a class of combinatorial objects that we call \textit{generalized labelled spaces} that have an underlying graph combinatorial structure. We then prove that the class of algebras associated to generalized labelled spaces is equivalent to the class of generalized Boolean dynamical systems. This alternative characterization acts as a Stone duality for generalized Boolean dynamical systems and shows that we can realize all such dynamical systems with an appropriate labelled graph.

	\section{(Relative) Generalized Boolean Dynamical System Algebras} \label{prelim}
		
	We first review the definition of a relative generalized Boolean dynamical system. For a more detailed overview, see \cite{CARLSEN2020124037}.
	
	\begin{definition}[{\cite[Section~2.2]{Castro2022CalgebrasOG}}]
		$\mathcal B$ is a \textit{generalized Boolean algebra} if it is a distributive lattice with relative complements and a minimal element $\emptyset$. Morphisms of generalized Boolean algebras are morphisms that preserve the minimal element and all operations. A subset $\mathcal I \subseteq \mathcal B$ is called an \textit{ideal} if $a, b \in \mathcal I \Rightarrow a\vee b \in \mathcal I$ and $a \in \mathcal I, b \in \mathcal B \Rightarrow a \wedge b \in \mathcal I$. $\mathcal B$ has an order defined by $a \leq b \Leftrightarrow a \wedge b = a$. 
	\end{definition}
	
	\begin{remark}
		We often think of elements of $\mathcal B$ as subsets of a set. Hence, we use capital letters (e.g. $A, B$) to refer to elements in $\mathcal B$ and use notation $A \cup B$, $A \wedge B$, $A \setminus B$, and $A \subseteq B$ to refer to joins, meets, relative complements, and the order respectively. This intuition can be justified with Stone duality \cite{Lawson2011NonCommutativeSD}.
	\end{remark}
	\begin{definition}[{\cite[Section~2.3]{CARLSEN2020124037}}]
		Let $\mathcal I \subseteq \mathcal B$ be an ideal. There is an equivalence relation on $\mathcal B$ defined by \[A \sim B \Leftrightarrow \exists U \in \mathcal I \text{ such that } A \cup U = B \cup U\] For $A \in \mathcal B$, we refer to its equivalence class as $[A]_{\mathcal I}$. The set of all equivalence classes $\mathcal B/\mathcal I = \{[A]_{\mathcal I} \colon A \in \mathcal B\}$ is also a generalized Boolean algebra with the obvious operations on representatives.
	\end{definition}
	\begin{definition}[{\cite[Definition~3.2]{CARLSEN2020124037}}]
		A \textit{relative generalized Boolean dynamical system} is a pentuplet \[(\mathcal B, \mathcal L, \theta, \mathcal I, \mathcal J)\] with the following properties:
		
		$\mathcal B$ is a generalized Boolean algebra. $\mathcal L$ is a set of characters $\alpha \in \mathcal L$.  $\theta$ is a set indexed by $\alpha \in \mathcal L$ where each $\theta_{\alpha}: \mathcal B \rightarrow \mathcal B$ is a morphism of generalized Boolean algebras. Define $\mathcal F_{\alpha} \coloneqq \{A \in \mathcal B \colon \exists B \in \mathcal B, A \subseteq \theta_\alpha(B)\}$. $\mathcal I$ is a set of ideals of $\mathcal B$ indexed by $\alpha \in \mathcal L$ such that $\mathcal F_{\alpha} \subseteq \mathcal I_{\alpha}$.
	 	 
 	 	For every $A \in \mathcal B$, define $\Delta^{(\mathcal B, \mathcal L, \theta, \mathcal I)}_A \coloneqq \{\alpha \in \mathcal L \mid \theta_{\alpha}(A) \neq \emptyset\}$. Define the ideal of \textit{regular} sets $\mathcal B^{(\mathcal B, \mathcal L, \theta, \mathcal I)}_{\text{reg}} \coloneqq \{A \in \mathcal B \colon \forall B \subseteq A,  0 < | \Delta_B| < \infty\}$. All other sets $B \in \mathcal B$ are called \textit{singular}. Define $\mathcal B^{(\mathcal B, \mathcal L, \theta, \mathcal I)}_{\text{sink}} = \{B \in \mathcal B \colon |\Delta_B| = 0\}$.
		
		$\mathcal J$ is an ideal of $\mathcal B^{(\mathcal B, \mathcal L, \theta, \mathcal I)}_{\text{reg}}$. When the dynamical system is clear, we use the notation $\Delta_A$ and  $\mathcal B_{reg}$. When $\mathcal J = \mathcal B^{(\mathcal B, \mathcal L, \theta, \mathcal I)}_{\text{reg}}$, we drop $\mathcal J$ and call the quadruplet $(\mathcal B, \mathcal L, \theta, \mathcal I)$ a \textit{generalized Boolean dynamical system}.
	\end{definition}

	\begin{definition} 
		Let $\mathcal L^{\ast}$ be the set of finite length (possibly empty) words with characters from $\mathcal L$. Define $\theta_{\omega}: \mathcal B \rightarrow \mathcal B$ to be the identity and $\mathcal I_{\omega} = \mathcal B$. For $\omega \neq \alpha = \alpha_1 \ldots \alpha_n \in \mathcal L^{\ast}$, define $\theta_{\alpha} = \theta_{\alpha_n} \circ \ldots \circ \theta_{\alpha_1}$ and $\mathcal I_{\alpha} = \{A \in \mathcal B \colon A \subseteq \theta_{\alpha_2 \ldots \alpha_n}(B) \text{ for some } B \in \mathcal I_{\alpha_1}\}$. 
	\end{definition}

	We now define an algebra associated to a relative generalized Boolean dynamical system. Throughout, let $(\mathcal B, \mathcal L, \theta, \mathcal I, \mathcal J)$ be a relative generalized Boolean dynamical system and let $R$ be a unital commutative ring.
	
	\begin{definition} \label{definition:boolrelations}
	Define the $R$-algebra $L_R(\mathcal B, \mathcal L, \theta, \mathcal I, \mathcal J)$ to be the $R$-algebra generated by symbols $\{p_A\}_{A \in \mathcal B} \cup \{s_{\alpha, A}, s_{\alpha, A}\}_{\alpha \in \mathcal L, B \in \mathcal I_{\alpha}}$ and the following relations:
	
	\begin{enumerate}
		\item $p_{\emptyset} = 0, p_{A\cap A'} = p_A p_{A'}, p_{A \cup A'} = p_A + p_{A'} - p_{A \cap A'}$.
		\item $p_{A}s_{\alpha, B} = s_{\alpha, B} p_{\theta_{\alpha}(A)}$ and $s_{\alpha, B}^{\ast} p_A = p_{\theta_{\alpha}(A)} s_{\alpha, B}^{\ast}$
		\item $s_{\alpha, B}^{\ast} s_{\alpha', B'} = \delta_{\alpha, \alpha'} p_{B \cap B'}$
		\item $s_{\alpha, B}p_{B'} = s_{\alpha, B \cap B'}$ and $p_{B'} s^{\ast}_{\alpha, B}= s^{\ast}_{\alpha, B \cap B'}$
		\item $p_A = \sum_{\alpha \in \Delta_A} s_{\alpha, \theta_{\alpha}(A)} s_{\alpha, \theta_{\alpha}(A)}^{\ast}$ for all $A \in \mathcal J$.
	\end{enumerate}
	
	For a generalized Boolean dynamical system $(\mathcal B, \mathcal L, \theta, \mathcal I)$, we define $L_R(\mathcal B, \mathcal L, \theta, \mathcal I) \coloneqq L_R(\mathcal B, \mathcal L, \theta, \mathcal I, \mathcal B_{\text{reg}})$.
	
	\end{definition}
	
	\begin{example} \label{example:labelledspace}
	As in the $C^{\ast}$-algebra case \cite[Example~4.2]{CARLSEN2020124037}, generalized Boolean dynamical system algebras are a generalization of labelled Leavitt path algebras associated to weakly left-resolving normal labelled spaces, which themselves are generalizations of Leavitt path algebras.
	\end{example}

	\begin{theorem} \label{theorem:zgrading}
		$L_R(\mathcal B, \mathcal L, \theta, \mathcal I, \mathcal J)$ has a $\mathbb Z$-grading given by \[s_{\alpha, A} \mapsto 1, s_{\alpha, A}^{\ast} \mapsto -1 \text{ for } \alpha \in \mathcal L \text{ and } A \in \mathcal I_{\alpha}\]
		\[p_A \mapsto 0 \text{ for } A \in \mathcal B \]
	\end{theorem}

	\begin{proof}
		The proof is standard and is found in \cite[Proposition~3.8]{Boava2021LeavittPA} for labelled Leavitt path algebras and in \cite[Section~2.1]{Abrams2017} for Leavitt path algebras.
	\end{proof}
	
	The following lemmas are standard and used in computations throughout the paper.

	\begin{lemma}[{\cite[Lemma~3.5]{Boava2021LeavittPA}}] \label{lemma:disjointsplit} Let $x = \mathrm{span}_R \{p_A \colon A \in \mathcal B\}\subseteq L_R(\mathcal B, \mathcal L, \theta, \mathcal I, \mathcal J)$ with \[x = \sum_{i=1}^n r_i p_{A_i}\]
		
	For $A_i \in \mathcal B$ and $r_i \in R$. Then there exists a family of pairwise disjoint sets $\{C_1, \ldots, C_m\}$ and coefficients $s_1, \ldots, s_m \in R$ such that \[x = \sum_{i=1}^m s_i p_{C_i}\] and for all $j \in [1, m]$, there exists some $i \in [1, n]$ such that $C_j \subseteq A_i$.
	\end{lemma}

	\begin{proof}
		The proof is the same as \cite[Lemma~3.5]{Boava2021LeavittPA}.
	\end{proof}

	\begin{lemma} \label{lemma:qcalc} Define the symbol $q_A \coloneqq p_A - \sum_{\alpha \in \Delta_A} s_{\alpha, \theta_{\alpha}(A)} s^{\ast}_{\alpha, \theta_{\alpha}(A)}$ whenever $|\Delta_A| < \infty$. Then $p_Aq_B = q_B p_A = q_{B \cap A}$ for $B \in \mathcal B_{\text{reg}}$ and $A \in \mathcal B$.	
	\end{lemma}
	\begin{proof}
		$p_A q_B = \sum_{\alpha \in \Delta_B} p_A s_{\alpha, \theta_{\alpha}(B)}s^{\ast}_{\alpha, \theta_{\alpha}(A)}$. This is equal to \[\sum_{\alpha \in \Delta_B} s_{\alpha, \theta_{\alpha}(B)} p_{\theta_{\alpha}(A)}s^{\ast}_{\alpha, \theta_{\alpha}(A)}\] It's not hard to show that this is also equal to $q_B p_A$. Furthermore, we have that $p_{\theta_{\alpha}(A)}^2 = p_{\theta_{\alpha}(A)}$ so both sums are equal to \[\sum_{\alpha \in \Delta_B} s_{\alpha, \theta_{\alpha}(B)} p_{\theta_{\alpha}(A)}p_{\theta_{\alpha}(A)}s^{\ast}_{\alpha, \theta_{\alpha}(A)}\] Using the fact that $\theta_{\alpha}$ preserves intersections, this is equal to $\sum_{\alpha \in \Delta_B} s_{\alpha, \theta_{\alpha}(A \cap B)} s^{\ast}_{\alpha, \theta_{\alpha}(A\cap B)}$. Now note that $\Delta_{A\cap B} \subseteq \Delta_{A}$ and for all $\alpha \in \Delta_A \setminus \Delta_{A\cap B}$ we have $\theta_{\alpha}(A \cap B) = \emptyset$ so any term associated with $\alpha \in \Delta_{A \cap B} \setminus \Delta_a$ in the sum is $0$. Hence, we get that the sum is $\sum_{\alpha \in \Delta_{A \cap B}} s_{\alpha, \theta_{\alpha}(A \cap B)} s^{\ast}_{\alpha, \theta_{\alpha}(A\cap B)} = q_{A \cap B}$.
	\end{proof}

	\section{Realization as a Partial Skew Group Ring} \label{partialskewgroup}
	Using the inverse semigroup associated to a generalized Boolean dynamical system in \cite[Section 3]{DECASTRO2023126662}, we realize our algebra $L_R(\mathcal B, \mathcal L, \theta, \mathcal I, \mathcal J)$ as a partial skew group ring. This section relies heavily on and uses notation from \cite{zhang2025partialactionsgeneralizedboolean}.
	
	\begin{definition}
		For a generalized Boolean dynamical system $(\mathcal B, \mathcal L, \theta, \mathcal I)$ define its inverse semigroup  \[S_{(\mathcal B, \mathcal L, \theta, \mathcal I)}= \{(\alpha, A, \beta) \colon \alpha, \beta \in \mathcal L^{\ast} \text{ and } \emptyset \neq A \in \mathcal I_{\alpha} \cap \mathcal I_{\beta}\}\]
		
		Multiplication is defined as 
		\[(\alpha, A, \beta)\cdot(\gamma, B, \delta) = \begin{cases}
			(\alpha, A \cap B, \delta) & \text{ if } \beta = \gamma \\
			(\alpha\gamma', \theta_{\gamma'}(A) \cap B, \delta) &\text{ if } \gamma = \beta \gamma' \text{ and } \theta_{\gamma'}(A) \cap B \neq \emptyset \\
			(\alpha, A \cap \theta_{\beta'}(B), \delta \beta') & \text{ if } \beta = \gamma \beta' \text{ and } A \cap \theta_{\beta'}(B) \neq \emptyset \\
			0 & \text{otherwise}
		\end{cases}\]
	with $(\alpha, A, \beta)^{\ast} = (\beta, A, \alpha)$ and the set of idempotents \[E = \{(\alpha, A, \alpha) \colon \alpha \in \mathcal L^{\ast} \text{ and } \emptyset \neq A \in \mathcal I_{\alpha}\}\cup\{0\}\]

	The semilattice of idempotents has natural order \[(\alpha, A, \alpha) \leq (\beta, B, \beta) \Leftrightarrow \alpha = \beta \alpha' \text { and } A \subseteq \theta_{\alpha'}(B)\]
		
	The inverse semigroup has grading $\varphi: S^{\times}_{(\mathcal B, \mathcal L, \theta, \mathcal I)} \rightarrow \mathbb F[\mathcal L]$ on the free group generated by $\mathcal L$ by taking \[(\alpha, A, \beta) \mapsto \alpha \beta^{-1}\] where we view $\alpha = \alpha_1 \alpha_2\ldots \alpha_n \in \mathbb F[\mathcal L]$. This grading makes the inverse semigroup strongly $E^{\ast}$-unitary.
		
	In the following we assume that representations for $g \in \mathbb F[\mathcal L]$ are reduced.
	The sets $\varphi^{-1}(g)$ are:
	\[\varphi^{-1}(g) = \begin{cases} 
		\{(p_1p, A, p_2p) \colon  p \in \mathcal L^{\ast}, \emptyset \neq A \in \mathcal I_{p_1p} \cap \mathcal I_{p_2p}\}& \text{ if } g = p_1 p_2^{-1} \text{ where } p_1, p_2 \in \mathcal L^{\ast} \\
		\emptyset & \text{otherwise}
	\end{cases}\]
	
	The sets $E_{g}$ are:
	\[E_{g} = \begin{cases} 
		\{(p_1p, A, p_1p) \colon  p \in \mathcal L^{\ast}, \emptyset \neq A \in \mathcal I_{p_1p} \cap \mathcal I_{p_2p}\} \cup \{0\} & \text{ if } g = p_1 p_2^{-1} \text{ where } p_1, p_2 \in \mathcal L^{\ast} \\
		\{0\} & \text{otherwise}
	\end{cases}\]
	
	The maps $\phi_{g^{-1}}: E_g \rightarrow E_{g^{-1}}$ act in the following manner for $g = p_1p_2^{-1}$:
	\[\phi_{g^{-1}}(p_1p, A, p_1p) = (p_2p, A, p_2p)\]
	
	The inverse semigroup $(S, \varphi)$ is orthogonal and semi-saturated.

	The isomorphism from the labelled Leavitt path algebra $L_R(\mathcal B, \mathcal L, \theta, \mathcal I)$ with generators $\{p_B\}_{B \in \mathcal B} \cup \{s_{\alpha, A}, s_{\alpha, A}^{\ast}\}_{\alpha \in \mathcal L, A \in \mathcal I_{\alpha}}$ to $L_R(S_{(\mathcal B, \mathcal L, \theta, \mathcal I)})$ is given by the map on generators 		\[p_B \mapsto (\omega, B, \omega)\delta_e, s_{a, A} \mapsto (a, A, a)\delta_a, \text{ and } s^{\ast}_{a, A} \mapsto (\omega, A, \omega)\delta_{a^{-1}}\]
	
	\end{definition}

	The proofs of all the claims in our definition are very similar to \cite[Section~5.2]{zhang2025partialactionsgeneralizedboolean}. Hence, we omit the proofs of all parts besides the final isomorphism, which we prove below.

	\begin{theorem} \label{theorem:skewiso}
	\[L_R(\mathcal B, \mathcal L, \theta, \mathcal I) \cong L_R(S_{(\mathcal B, \mathcal L, \theta, \mathcal I)})\] with maps on generators \[p_B \mapsto (\omega, B, \omega)\delta_e, s_{a, A} \mapsto (a, A, a)\delta_a, \text{ and } s^{\ast}_{a, A} \mapsto (\omega, A, \omega)\delta_{a^{-1}}\] where we use the grading $(\alpha, A, \beta) \mapsto \alpha \beta^{-1}$ to compute $L_R(S_{(\mathcal B, \mathcal L, \theta, \mathcal I)})$.
	\end{theorem}

	\begin{proof}
	This proof will extensively use, without explicit reference, calculations for $L_R(S_{(\mathcal B, \mathcal L, \theta, \mathcal I)})$ found in \cite[Lemma~4.15]{zhang2025partialactionsgeneralizedboolean}. We first check that the mapping is well-defined by showing that the images obey the relations in Definition~\ref{definition:boolrelations}. 

	\begin{enumerate}
	\item Same as (1) in \cite[Lemma~5.7]{zhang2025partialactionsgeneralizedboolean}.
	\item \[p_A s_{a, B} \mapsto ((\omega, A, \omega) \delta_e) ((a, B, a)\delta_a) = (a, \theta_a(A) \cap B, a) \delta_a\] \[s_{a, B} p_{\theta_a(A)} \mapsto ((a, B, a) \delta_a)((\omega, \theta_a(A), \omega)\delta_e) = (a, B \cap \theta_a(A), a) \delta_a \] The other equality is much of the same.
	\item \[s_{a, B}^{\ast} s_{a', B'} \mapsto ((\omega, B, \omega) \delta_{a^{-1}}) ((a', B', a')\delta_{a'})\]\[=\delta_{aa'} (\omega, B, \omega) (\omega, B', \omega)\delta_e = \delta_{aa'}(a, B \cap B', a) \delta_e\] which is the image of $\delta_{aa'} p_{B \cap B'}$. In this proof, we used the fact that if $a \neq a'$, then $E_{a^{-1}a'} = \emptyset$ so the product must be $0$ whenever $a \neq a'$.
	\item \[s_{a, B} p_{B'} \mapsto ((a, A, a) \delta_a)((\omega, B, \omega) \delta_e)\]\[=\phi_a((\omega, A, \omega)(\omega, B, \omega))\delta_a = (a, A \cap B, a) \delta_a\] which is the image of $s_{a, A \cap B}$. The other equality is much of the same.
	\item The proof is same as \cite[Lemma~5.7]{zhang2025partialactionsgeneralizedboolean}.
	\end{enumerate}

	With the relations established and using the universality of $L_R(\mathcal B, \mathcal L, \theta, \mathcal I)$, we find that there is an $R$-algebra morphism to $L_R(S_{(\mathcal B, \mathcal L, \theta, \mathcal I)})$ induced by the given map on the generators. We now prove that the morphism is surjective using \cite[Corollary~4.17]{zhang2025partialactionsgeneralizedboolean} in a similar way as \cite[Lemma~5.8]{zhang2025partialactionsgeneralizedboolean}. 
	
	The difference in this proof is that our covers on $E_{a^{-1}}$ and $E_a$ are defined as $C_{a^{-1}} = \{(\omega, A, \omega) \colon 0 \neq A \in \mathcal I_a\}$ and $C_a = \{(a, A, a) \colon 0 \neq A \in \mathcal I_a \}$. Their respective values in the algebra are clearly in the image of our morphism as they are the images of $s_{a, A}^{\ast}$ and $s_{a, A}$ respectively.
	
	Now let $(\alpha, A, \alpha) \in E$. We want to show that $(\alpha, A, \alpha) \delta_e$ is in the image of our map. If $\alpha = \omega$, we have that $(\omega, A, \omega)\delta_e$ is in the image of $p_A$. Otherwise, write $\alpha = \alpha_1 \ldots \alpha_n$. In this case, as in \cite[Definition~3.6]{CARLSEN2020124037}, we define \[s_{\alpha, A} = s_{\alpha_1, B} s_{\alpha_2, \theta_{\alpha_2}(B)} s_{\alpha_3, \theta_{\alpha_2\alpha_3}(B)} \ldots s_{\alpha_n, A}\] It's not hard to calculate then that $s_{\alpha, A} \mapsto (\alpha, A, \alpha) \delta_{\alpha}$ and $s_{\alpha, A}^{\ast} \mapsto (\omega, A, \omega) \delta_{\alpha^{-1}}$. It's not hard to show then that $s_{\alpha, A} s_{\alpha, A}^{\ast} \mapsto (\alpha, A, \alpha) \delta_e$. This finishes the proof that the image contains a set of generators for $L_R(S_{(\mathcal B, \mathcal L, \theta, \mathcal I)})$, so the map is surjective.
	
	Similarly to \cite[Proposition~4.11 and Theorem~5.6]{Boava2021LeavittPA}, the proof of injectivity follows immediately from the graded uniqueness theorem for $L_R(\mathcal B, \mathcal L, \theta, \mathcal I)$, to be proved in Theorem~\ref{gradednonrel}. After the graded uniqueness theorem, we prove that the map is injective in the same way as in \cite[Lemma~5.8]{zhang2025partialactionsgeneralizedboolean}. However, the proof of the graded uniqueness theorem needs the below Corollary~\ref{corollary:nonzero}. Importantly, the corollary doesn't depend on the injectivity of the map, so the argument is not circular.

	\end{proof}
	
	\begin{corollary} \label{corollary:nonzero}
	Let $0 \neq r \in R$ be arbitrary. Then, for all $B \in \mathcal B$ and $\alpha \in \mathcal L^{\ast}$ and $\emptyset \neq A \in \mathcal I_\alpha$ we have that \[rp_B \neq 0\]\[rs_{\alpha, A} \neq 0\]
	\end{corollary}

	\begin{proof}
		From Theorem~\ref{theorem:skewiso} (without needing injectivity), there is a map $L_R(\mathcal B, \mathcal L, \theta, \mathcal I) \rightarrow L_R(S_{(\mathcal B, \mathcal L, \theta, \mathcal I)})$. We calculated that \[rp_B \mapsto r (\omega, B, \omega)\delta_e\] \[rs_{\alpha, A} \mapsto r (\alpha, A, \alpha) \delta_{\alpha}\] all of which are non-zero in $L_R(S_{(\mathcal B, \mathcal L, \theta, \mathcal I)})$ by \cite[(11) Lemma 4.15]{zhang2025partialactionsgeneralizedboolean}. Hence, $rp_B$ and $rs_{\alpha, A} \neq 0$.
	\end{proof}

	The follow theorem is analagous to a non-relative version of \cite[Proposition~6.3]{CARLSEN2020124037}. 

	\begin{theorem}
		Let $(\mathcal B, \mathcal L, \theta, \mathcal I)$ be a generalized Boolean dynamical system. Let $(\mathcal B, \mathcal L, \theta, \mathcal B)$ refer to the generalized Boolean dynamical system where all ideals are the generalized Boolean algebra $\mathcal B$. Then, $L_R(\mathcal B, \mathcal L, \theta, \mathcal I)$ and $L_R(\mathcal B, \mathcal L, \theta, \mathcal B)$ are Morita equivalent.
	\end{theorem}
	\begin{proof}
		To prove this theorem, we use the fact that our inverse semigroup realizes our algebra (Theorem~\ref{theorem:skewiso}) and apply \cite[Theorem~5.8]{zhang2025moritaequivalencesubringsapplications} to derive the Morita equivalence.

		For $\alpha \in \mathcal L^{\ast}$, let $\mathcal B_{\alpha}$ refer to the ideals associated to the system $(\mathcal B, \mathcal L, \theta, \mathcal B)$. Because $\mathcal I_a \subseteq \mathcal B_a = \mathcal B$ for all $a \in \mathcal L$, we find that $\mathcal I_{\alpha} \subseteq \mathcal B_{\alpha}$. Namely, there is an obvious injection \[S_{(\mathcal B, \mathcal L, \theta, \mathcal I)} \hookrightarrow S_{(\mathcal B, \mathcal L, \theta, \mathcal B)}\]\[(\alpha, B, \beta) \rightarrow (\alpha, B, \beta) \text{ for } B \in \mathcal I_{\alpha} \cap \mathcal I_{\beta}\]

		It remains to show that this inverse subsemigroup satisfies the conditions for \cite[Theorem~5.8]{zhang2025moritaequivalencesubringsapplications}. We will show these each separately.

		\begin{enumerate}
			\item Obvious
			\item Let $(\alpha, A, \alpha) \in E_{(\mathcal B, \mathcal L, \theta, \mathcal I)}$ and $(\beta, B, \beta) \in  E_{(\mathcal B, \mathcal L, \theta, \mathcal B)}$. This means that $\alpha, \beta \in \mathcal L^{\ast}$ and $A \in \mathcal I_{\alpha}$ while $B \in \mathcal B_{\alpha}$. Assume that $(\beta, B, \beta) \leq (\alpha, A, \alpha)$ which means that $\beta = \alpha \gamma$ for some $\gamma \in \mathcal L^{\ast}$ and $B \subseteq \theta_{\gamma}(A)$. Because $A \in \mathcal I_{\alpha}$, we know that $\theta_{\gamma}(A) \in \mathcal I_{\alpha\gamma} = \mathcal I_{\beta}$. This implies that $B \in \mathcal I_{\beta}$ so $(\beta, B, \beta) \in E_{(\mathcal B, \mathcal L, \theta, \mathcal I)}$ and thus the meet subsemilattice is closed downwards.
			\item We will show that for any $(\alpha, A, \alpha) \in E_{(\mathcal B, \mathcal L, \theta, \mathcal I)}$ and $(\beta, B, \gamma) \in S_{(\mathcal B, \mathcal L, \theta, \mathcal B)}$, writing $(\delta, C, \delta') = (\alpha, A, \alpha)(\beta, B, \gamma)$ we have that $C \in \mathcal I_{\delta}$ (assuming that everything is non-zero). The full condition is not hard to prove from this, but requires a lot of casework. We will prove our simpler statement with two cases. If $\alpha = \beta \psi$ for some $\psi$, we would have that $\delta = \alpha \psi$ and that $C \subseteq \theta_{\psi}(A) \in \mathcal I_{\alpha\psi} = \mathcal I_{\delta}$, so we are done. In the other case, if $\beta = \alpha \psi$ for some $\psi$, we have that $\delta = \alpha$ and $C \subseteq A \in \mathcal I_{\alpha}$, so again we are done.

			\item It is easy to prove that the projections $\{p_A\}$ generate $L_R(\mathcal B, \mathcal L, \theta, \mathcal B)$ as a two-sided ideal. For any $A \in \mathcal B$, we have that $p_A \in L_R(\mathcal B, \mathcal L, \theta, \mathcal I) \subseteq L_R(\mathcal B, \mathcal L, \theta, \mathcal B)$ so there is no proper two-sided ideal that contains the subalgebra.
		\end{enumerate}
	\end{proof}

	\section{Relative to Non-Relative} \label{relativetonon}
	Throughout this section, fix a relative generalized Boolean algebra $(\mathcal B, \mathcal L, \theta, \mathcal I, \mathcal J)$. A part of our contribution is to extend some results in the non-relative case to the relative case. In \cite{CARLSEN2020124037}, for the case of $C^{\ast}$-algebras, this is done using some results for the representations of $C^{\ast}$-algebras. For the algebraic case, we will take a more direct approach.
	
	Our main tool will be an isomorphism between the $R$-algebra $L_R(\mathcal B, \mathcal L, \theta, \mathcal I, \mathcal J)$ with an $R$-algebra $L_R(\tilde{\mathcal B}, \mathcal L, \tilde{\theta}, \tilde{\mathcal I})$ for a generalized Boolean dynamical system $(\tilde{\mathcal B}, \mathcal L, \tilde{\theta}, \tilde{\mathcal I})$. We will then derive results in the relative case by applying this isomorphism and using results in the non-relative case. The construction of $(\tilde{\mathcal B}, \mathcal L, \tilde{\theta}, \tilde{\mathcal I})$ comes directly from \cite[Proposition~6.4]{CARLSEN2020124037}, which we repeat below.
	
	\[\tilde{\mathcal B} = \{(A, [B]_{\mathcal J}) \colon A, B \in \mathcal B \text{ and } [A]_{\mathcal B_{\text{reg}}} = [B]_{\mathcal B_{\text{reg}}}\}\] \[\tilde{\mathcal I}_{\alpha} = \{(A, [A]_{\mathcal J}) \colon A \in \mathcal I_{\alpha}\}\] \[\tilde{\theta}_{\alpha}(A, [B]_{\mathcal J}) = (\theta_{\alpha}(A), [\tilde{\theta}(A)]_{\mathcal J})\] where operations on $\tilde{\mathcal B}$ are defined pairwise and by choosing arbitrary representations for the second coordinate. It's easy to show that this is a generalized Boolean dynamical system and furthermore that $\tilde{\mathcal B}_{\text{reg}} = \{(A, \emptyset) \colon A \in \mathcal B_{\text{reg}}\}$.
		
	\begin{theorem} \label{theorem:relativeiso}
		\[L_R(\mathcal B, \mathcal L, \theta, \mathcal I, \mathcal J) \cong L_R(\tilde{\mathcal B}, \tilde{\mathcal L}, \theta, \tilde{\mathcal I})\] The isomorphism is realized by the maps $\varphi: L_R(\mathcal B, \mathcal L, \theta, \mathcal I, \mathcal J) \rightarrow L_R(\tilde{\mathcal B}, \mathcal L, \tilde{\theta}, \tilde{\mathcal I})$ and $\psi: L_R(\tilde{\mathcal B}, \mathcal L, \tilde{\theta}, \tilde{\mathcal I}) \rightarrow L_R(\mathcal B, \mathcal L, \theta, \mathcal I, \mathcal J)$ with the following maps on the generators:
		\[\varphi(p_A) = p_{(A, [A]_{\mathcal J})}, \varphi(s_{a, B}) = s_{a, (B, [B]_{\mathcal J})}\]
		\[\psi(p_{(A, [B]_{\mathcal J})}) = p_A + q_C - q_D, \psi(s_{(a, (A, [A]_{\mathcal J}))}) = s_{a, A}\]
		where in the last statement $C, D \in \mathcal B_{\text{reg}}$ are any such sets such that $A \cup C = B \cup D$ and $A \cap C = \emptyset = B \cap D$.
		
	\end{theorem}
	
	\begin{proof}
		The proof that the resulting algebras are isomorphic is the same as in \cite[Proposition~6.4]{CARLSEN2020124037}, so we omit it for the sake of brevity.
	\end{proof}

	\begin{corollary} \label{corollary:nonzerorel} Let $0 \neq r \in R$ be arbitrary. Then,
		\[rp_B \neq 0 \text{ for any } B \in \mathcal B\]
		\[rs_{\alpha, A} \neq 0 \text{ for any } \alpha \in \mathcal L^{\ast}, \emptyset \neq A \in \mathcal I_{\alpha}\]
		\[rq_B \neq 0 \text{ for any } B \in \mathcal B_{\text{reg}} \setminus \mathcal J\]
	\end{corollary}
	\begin{proof}
		The first two statements follow from applying the isomorphism $\varphi$ in Theorem~\ref{theorem:relativeiso} and using Corollary~\ref{corollary:nonzero}. To see the third statement, let $B \in \mathcal B_{\text{reg}} \setminus \mathcal J$ and note that $(\emptyset, [B]_{\mathcal J}) \neq \emptyset$ in $\tilde{\mathcal B}$ so $p_{(\emptyset, [B]_{\mathcal J})} \neq 0$. Using $C = B$ and $D = \emptyset$, we calculate that $\psi(p_{(\emptyset, [B]_{\mathcal J})}) = q_B$, which is not equal to $0$ because $\psi$ is an isomorphism.
	\end{proof}

	\section{Cuntz-Pimsner Algebras and Graded Uniqueness Theorems} \label{section:gradedunique}
	To prove that the generalized Boolean Dynamical systems algebras $L_R(\mathcal B, \mathcal L, \theta, \mathcal I)$ are also Cuntz-Pimsner algebras, we modify the proof in \cite[Section~5]{Boava2021LeavittPA}. We then apply a graded uniqueness theorem for Cuntz-Pimsner algebras to derive one for generalized Boolean dynamical system algebras. Finally, using the isomorphism from Theorem~\ref{theorem:relativeiso}, we derive a graded uniqueness theorem for relative generalized Boolean dynamical systems.
	
	For ease, we restate the process from \cite[Section~5]{Boava2021LeavittPA}. Our goal will be to use \cite[Theorem~3.1]{CLARK201982} to realize a generalized Boolean dynamical system algebra as a Cuntz-Pimsner algebra and then apply the graded uniqueness theorem for Cuntz-Pimsner algebras from \cite[Corollary~5.4]{Carlsen2008AlgebraicCR}.
	
	The relevant definitions are found in \cite{Carlsen2008AlgebraicCR, CLARK201982}. We now set up some notation from \cite[Theorem~3.1]{CLARK201982}. Let $S$ be a ring, $M$ a left $S$-module, and $I$ a subset of $M$. The \textit{left annihalator} of $I$ by $S$ defined by \[\mathrm{Ann}_S(I) = \{r \in S \colon rx = 0 \text { for all } x \in I\}\] is a left ideal of $S$. If $I$ is a sub-module of $M$, then $\mathrm{Ann}_S(I)$ is a two-sided ideal of $S$. Let $J$ be a two-sided ideal of a ring $S$. We define \[J^{\bot} = \{r \in S \colon ry = yr = 0 \text{ for all } y \in J\}\]
	
\begin{theorem}\cite[Theorem~3.1]{CLARK201982} \label{ssystem}
	Let $A=\bigoplus_{i\in \mathbb Z}A_i$ be a $\mathbb Z$-graded ring, $S$ a subring of $A_0$, and $I\subseteq A_1$ and  $J\subseteq  A_{-1}$ additive subgroups such that:
	\begin{enumerate}
		\item $SI, IS \subseteq I$, $SJ, JS \subseteq J$ and 
		$JI\subseteq S$;
		
		\medskip
		
		\item For any finite subset $\{i_1,\dots,i_n\}\subseteq I$ there is an element $a$ in $IJ$ such that $a i_l=i_l$ for each $1\leq l \leq n$, and for 
		any finite subset $\{j_1,\dots,j_m\}\subseteq J$ there is an element $b$ in $IJ$ such that $j_l b=j_l$ for each $1\leq l \leq m$;
		
		\medskip
		
		\item For $x\in \mathrm{Ann}_S(I)^\bot$ and $a\in IJ$, if $x-a \in \mathrm{Ann}_{A_0}(I)$, then $a\in S$;
		
		\medskip 
		
		\item $\mathrm{Ann}_S(I)\cap \mathrm{Ann}_S(I)^\bot=\{0\}$.
	\end{enumerate}

	Then there exists an $S$-bimodule homomorphism $\psi:J\otimes_S I \rightarrow S$ such that $\psi(j\otimes_S i)=ji$ for each $j\in J, i\in I$, and $(J,I,\psi)$ is an $S$-system. Furthermore, there is a graded isomorphism from the Cuntz-Pimsner algebra $\mathcal O_{(J,I,\psi)}$ of the $S$-system $(J,I,\psi)$ to the subring of $A$ generated by $S,I,J$. 
	\end{theorem}

Our proof of the below theorem mirrors the proof of \cite[Theorem~5.3]{Boava2021LeavittPA}

\begin{theorem} \label{cuntzrealize}
	Let $(\mathcal B, \mathcal L, \theta, \mathcal I)$ be a generalized Boolean dynamical system and $R$ a unital commutative ring. Then there exists a subring $S \subseteq L_R(\mathcal B, \mathcal L, \theta, \mathcal I)$ and an $S$-system $(J, I, \psi)$ such that $L_R(\mathcal B, \mathcal L, \theta, \mathcal I)$ is graded isomorphic to the Cuntz-Pimsner algebra $\mathcal O_{(J, I, \psi)}$.
\end{theorem}

\begin{proof}
	
	We give $L_R(\mathcal B, \mathcal L, \theta, \mathcal I)_0$ the $\mathbb Z$-grading from Theorem~\ref{theorem:zgrading}. Let $S = \text{span}_R \{p_A \colon A \in \mathcal B\}$, $I = \text{span}_R \{s_{a, A} \colon a \in \mathcal L, A \in \mathcal I_a\}$ and $J = \text{span}_R\{s^{\ast}_{a, A} \colon a \in \mathcal L, A \in \mathcal I_a\}$. It's clear that $S \subseteq L_R(\mathcal B, \mathcal L, \theta, \mathcal I)$ is a subring and is contained in the homogenous component $L_R(\mathcal B, \mathcal L, \theta, \mathcal I)_0$. We now verify the conditions of Theorem~\ref{ssystem}.
	
	Condition (1) is easy from the relations on our algebra.
	
	To prove Condition (2), let $\{i_1, \ldots, i_n\} \subseteq I$. We write $i_l = \sum_{j=1}^{n_l} \lambda^l_js_{a^l_j, A^l_j}$ where $A_l^j \in \mathcal I_{a_l^j}$ and $\lambda^l_j \in R$. Let $F = \bigcup_{l=1}^n \{a^l_1, \ldots, a^l_{n_l}\}$ and for $a \in F$ let $A_b = \bigcup_{l=1}^n \bigcup_{i=1}^{n_l} \delta_{a_l^j, b} A^l_j$.  It's clear that $A_b \in \mathcal I_b$. Now take $a = \sum_{b \in F} s_{b, A_b} s_{b, A_b}^{\ast} \in IJ$ and it's not difficult to see that $ai_l = i_l$, so we are done. The second part of Condition (2) follows analagously.
	
	It remains to prove Conditions (3) and (4). To do this, we set up a series of claims.
	
	{\bf Claim 1.}	We have that \[\text{Ann}_S(I) = \text{span}_R\{p_A \colon A \in \mathcal B_{\text{sink}}\}\]
	
	\begin{proof}
		Consider some element $x \in I$ and let $x = \sum_{i=1}^n \gamma_i s_{a_i, A_i}$.
		
		First let $s \in \text{span}_R\{p_A \colon A \in \mathcal B_{\text{sink}}\}$ with $s = \sum_{j=1}^m \lambda_j p_{A_j}$ for $\lambda_j \in R$ and $A_j \in \mathcal B_{\text{sink}}$. Consider one term of $sx$ which looks like $\lambda_j p_{A_j} \gamma_i s_{a_i, A_i} = \lambda_j \gamma_i s_{a_i, A_i} p_{\theta_{a_i}(A_j)}$. Because $A_j \in \mathcal B_{\text{sink}}$, we have that $\theta_{a_i}(A_j) = \emptyset$ so each term is equal to $0$ and hence $s \in \text{Ann}_S(I)$ so $\text{span}_R\{p_A \colon A \in \mathcal B_{\text{sink}}\} \subseteq \text{Ann}_S(I)$.
		
		To see the converse, let $s =  \sum_{j=1}^m \lambda_j p_{A_j}\in S$ with $\lambda_j \in R$ and $A_j \in \mathcal B$ be an annihalator of $I$. By Lemma~\ref{lemma:disjointsplit}, we assume that each $\{A_j\}_{j=1}^m$ is pairwise disjoint. Assume that there is some $k \in [1, m]$ where $A_k$ is not a sink, so there exists $a \in \mathcal L$ such that $\theta_a(A_k) \neq \emptyset$. Consider $s_{a, \theta_a(A_k)} \in I$. Because $S$ annihalates $I$, we have that $0 = s s_{a, \theta_a(A_k)} = \sum_{j=1}^m \lambda_j p_{A_j} s_{a, \theta_a(A_k)} = \sum_{j=1}^m \lambda_j s_{a, \theta_a(A_k) \cap \theta_a(A_j)} = \sum_{j=1}^m \lambda_j s_{a, \theta_a(A_k \cap A_j)} = \lambda_k s_{a, \theta_a(A_k)}$ where the last equality follows by the pairwise disjointness of $\{A_j\}_{j=1}^m$. Because $\theta_a(A_k) \neq \emptyset$, by Corollary~\ref{corollary:nonzero} this is not $0$, so we have a contradiction.
	\end{proof}

	{\bf Claim 2.} We have that \[\text{Ann}_S(I)^{\perp} = \text{span}_R \{p_A \colon A \in \mathcal B \text{ such that } \forall B \in \mathcal B_{\text{sink}} \text{ we have that } A \cap B = \emptyset\}\]
	\begin{proof}
		Note that if any $A$ satisfies that $A \cap B = \emptyset$ for any $B \in \mathcal B_{\text{sink}}$, then it's clear that by Claim 1 we have that $p_A s = 0$ for any $s \in \text{Ann}_S(I)$. Thus is obviously preserved by the $R$-span, so we find that $\text{Ann}_S(I)^{\perp} \supseteq \text{span}_R \{p_A \colon A \in \mathcal B \text{ such that } \forall B \in \mathcal B_{\text{sink}} \text{ we have that } A \cap B = \emptyset\}$ so we are left with the reverse direction.
		
		Let $s \in \text{Ann}_S(I)^{\perp}$. We write $s = \sum_{j=1}^m \lambda_j p_{A_j}$ for pairwise disjoint $A_j$. Assume that there is some $k \in [1, m]$ such that $A_k$ does not satisfy our required condition. Then there exists some $B \in \mathcal B_{\text{sink}}$ such that $A_k \cap B \neq \emptyset$. Without loss of generality, assume that $\emptyset \neq B \subseteq A_k$ (otherwise take $B \coloneqq B \cap A$ which works by our definition of sinks). Then consider $p_B \in \text{Ann}_S(I)$ and note $0 = sp_B$ because $s$ is in the complement of the annihalator. However, we also calculate that $sp_B = \lambda_k p_B$ by $B \subseteq A$ and the pairwise disjointness of $A_i$. But this is a contradiction because this is $0$ only when $B$ is the emptyset, and it is not the empty by construction, so we are done.
		
	\end{proof}
	
	{\bf Claim 3.} Let $a, b \in \mathcal L$ with $a \neq b$. Take any $A \in \mathcal I_a, B \in \mathcal I_b$ and $p \in S$, then $(s^{\ast}_{a, A} )p (s_{b, B}) = 0$. 

	\begin{proof}
		We have that $p = \sum_{j=1}^m \lambda_j p_{A_j}$. Looking again term by term, we end up with $s^{\ast}_{a, A} s_{b, B} p_{\theta_b(A_j)}$ for each term. The first part of this is $0$ because $a \neq b$, so all terms are $0$.
	\end{proof}

	{\bf Claim 4. } Let $a \in IJ$. We write $a = \sum_{i=1}^n \mu_i s_{a_i, X_i} s^{\ast}_{b_i, X_i}$ where no elements are $0$, $X_i \in \mathcal I_{a_i} \cap \mathcal I_{b_i}$ and if $a_j = a_k$ and $b_j = b_k$ we have that $X_j \cap X_k = \emptyset$.

	\begin{proof}
		We first prove that any element in $IJ$ has the given representation $\sum_{i=1}^n \mu_i s_{a_i, X_i} s^{\ast}_{b_i, X_i}$ and enforce the disjointness condition later.
		
		It suffices to prove the first part for the product of two elements in $a_1 \in I$ and $a_2 \in J$. A single term in $a_1a_2$ looks like $r s_{a, A} s^{\ast}_{b, B}$ for some $r \in R$, $A \in \mathcal I_{\alpha}$, and $B \in \mathcal I_{b}$. Recall that $s_{a, A \cap B} = s_{a, A} p_B$ and $s^{\ast}_{b, A \cap B} = p_B s^{\ast}_{b, A}$. Thus, $r s_{a, A} s^{\ast}_{b, B} = r s_{a, A} p_A p_B s^{\ast}_{b, B} = r s_{a, A \cap B} s^{\ast}_{b, A \cap B}$ which gives us our $X_i$.
		
		We now fix $a, b \in \mathcal L$ and consider individual sums of the form $\sum_{i=1}^n \lambda_i s_{a, X_i} s^{\ast}_{b, X_i}$. Let $X = \bigcup_{i=1}^n X_i$. We write our sum as $\sum_{i=1}^n s_{a, X_i} s^{\ast}_{b, X_i} = \sum_{i=1}^n s_{a, X} p_{X_i} s^{\ast}_{a, X} = s_{a, X} \left( \sum_{i=1}^n \lambda_i p_{X_i} \right) s^{\ast}_{a, X}$. We again use Lemma~\ref{lemma:disjointsplit} to make $\sum_{i=1} \lambda_i p_{X_i} = \sum_{j=1}^m \mu_j X'_j$ for disjoint $\{X'_j\}_{j=1}^m$. Multiplying back inside and taking $X_j \coloneqq X'_j$, we have our desired representation. 
	\end{proof}
	
	{\bf Claim 5. } Suppose that $x-a \in \text{Ann}_{L_R(\mathcal B, \mathcal L, \theta, \mathcal I)_0}(I)$ where $x \in \text{Ann}_S(I)^{\perp}$ and $a \in IJ$ with $a=\sum_{i=1}^n \mu_i s_{a_i, X_i} s^{\ast}_{b_i, X_i}$ as in Claim 4, then $a_i = b_i$ for all $i$.

	\begin{proof}
		Suppose for the sake of contradiction that there is some $k \in [1, n]$ where $a_k \neq b_k$. Consider the term $s^{\ast}_{a_k, X_k} (x-a) s_{b_k, X_k}$. Because $x-a$ annihalates $I$ and $s_{b_k, X_k} \in I$, we have that this entire term is $0$.
		
		Furthermore, note that by combining Claim 2 and Claim 3 and using the fact that $a_k \neq b_k$, we have that $s^{\ast}_{a_k, X_k} x s_{b_k, X_k} = 0$ because $x \in S$. Hence, we conclude that $s^{\ast}_{a_k, X_k} a s_{b_k, X_k} = 0$. We calculate $s^{\ast}_{a_k, X_k} a s_{b_k, X_k}$ to also be equal to \[\sum\limits_{i \in [1, n] \colon a_i = a_k, b_i = b_k} \mu_i p_{X_i \cap X_k} = \mu_k p_{X_k}\] by the assumption that any terms that match $a_j, b_j$ have disjoint $X_j$. However, by Corollary~\ref{corollary:nonzero} this is not $0$, which is a contradiction.
	\end{proof}

	Let $a \in IJ$ and $x \in \text{Ann}_S(I)^{\perp}$ such that $x-a \in \text{Ann}_{L_R(\mathcal B, \mathcal L, \theta, \mathcal I)_0}(I)$. By Claim 2, we can write \[x = \sum_{i=1}^m \lambda_i p_{A_i}\] where $A_i \in \mathcal B$ are pairwise disjoint with trivial intersection with all sets in $\mathcal B_{\text{sink}}$ and $\lambda_i \neq 0$. By Claim 5, we can write \[a = \sum_{i=1}^n \mu_i s_{a_i, B_i} s_{a_i, B_i}^{\ast}\] for $\emptyset \neq B_i \in \mathcal I_{a_i}$ and pairwise disjoint. Fix these representions of $a$ and $x$.
	
	{\bf Claim 6. }	Fix $l \in [1, m], k \in [1, n]$. If $\theta_{a_k}(A_l) \neq \emptyset$, then,
		
		\[\lambda_l p_{\theta_{a_k}(A_l)} = \sum\limits_{j \in [1, n] \colon a_j = a_k, \theta_{a_k}(A_l) \cap B_j \neq \emptyset} \mu_j p_{\theta_{a_k}(A_l) \cap B_j}\]      
	\begin{proof}
		Similar to the last claim, we have that \[s^{\ast}_{a_k, \theta_{a_k}(A_l)} (x-a) s_{a_k, \theta_{a_k}(A_l)}p_{\theta_{a_k}(A_l)} = 0 \Rightarrow s^{\ast}_{a_k, \theta_{a_k}(A_l)} xs_{a_k, \theta_{a_k}(A_l)} p_{\theta_{a_k}(A_l)} = s^{\ast}_{a_k, \theta_{a_k}(A_l)} a s_{a_k, A_l} p_{\theta_{a_k}(A_l)}\]
		
		If we expand the left-hand side, we find that it is equal to \[\sum_{j=1}^m \lambda_j s^{\ast}_{a_k, \theta_{a_k}(A_l)} p_{A_j} s_{a_k, \theta_{a_k}(A_l)} p_{\theta_{a_k}(A_l)} =  \sum_{j=1}^m \lambda_j s^{\ast}_{a_k, \theta_{a_k}(A_l)} s_{a_k, \theta_{a_k}(A_l)} p_{\theta_{a_k}(A_j)}p_{\theta_{a_k}(A_l)}\] \[=\sum_{j=1}^m \lambda_j p_{\theta_{a_k}(A_l)} p_{\theta_{a_k}(A_j \cap A_l)} = \lambda_j p_{\theta_{a_k}(A_l)}\]

		Expanding the right-hand side gives us \[\sum_{j=1}^n \mu_j s_{a_k, \theta_{a_k}(A_l)}^{\ast}s_{a_j, B_j}s_{a_j, B_j}^{\ast}s_{a_k, \theta_{a_k}(A_l)}p_{\theta_{a_k}(A_l)} = \sum_{j \colon a_j = a_k} \mu_j p_{\theta_{a_k}(A_l) \cap B_j}\] because the term at $j$ where $a_j \neq a_k$ is just $0$. Enforcing the $\theta_{a_k}(A_l) \cap B_j \neq \emptyset$ condition is trivial.
	\end{proof}
	
	{\bf Claim 7. } For all $k \in [1, n]$, we have that $B_k = \bigsqcup_{i=1}^m \theta_{a_k}(A_i) \cap B_k$.
	
	\begin{proof}
		$\bigsqcup_{i=1}^m \theta_{a_k}(A_i)$ is a disjoint union because $A_i$ are pairwise disjoint. Let $C_k = B_k \setminus \bigcup_{i=1}^m \theta_{a_k}(A_i)$. We will show that $C_k = \emptyset$. Note that $C_k \in \mathcal I_{a_k}$ because $B_k \in \mathcal I_{a_k}$.
		
		Hence, we write \[xs_{a_k, C_k} = \sum_{i=1}^m \lambda_i p_{A_i} s_{a_k, C_k} = \sum_{i=1}^m \lambda_i s_{a_k, C_k \cap \theta_{a_k}(A_i)} = 0\] Now define $D_k = \bigcup_{j \in [1, n] \colon a_j = a_k}^n B_j$. Note that $D_k \in \mathcal I_{a_k}$ so  $s_{a_k, D_k}$ is a valid element. We find that \[0 = s^{\ast}_{a_k, D_k}(x-a)s_{a_k, C_k} = -s^{\ast}_{a_k, D_k}as_{a_k, C_k} = -s^{\ast}_{a_k, D_k}\left( \sum_{j=1}^n \mu_j s_{a_j, B_j} s^{\ast}_{a_j, B_j} \right) s_{a_k, C_k} = \]
		\[-\sum_{j \in [1, n] \colon a_j = a_k} \mu_j p_{D_k \cap C_k \cap B_j} = -\sum_{j \in [1, n]\colon a_j = a_k} \mu_j p_{C_k \cap B_j} = -\mu_k p_{C_k}\] In the last equality, we use the fact that for $i\neq j$ such that $a_i = a_j$, then $B_i$ and $B_j$ are pairwise disjoint. For the second to last equality we use the fact that $B_j \subseteq D_k$ for any $j$ such that $a_j = a_k$. Thus, by Corollary~\ref{corollary:nonzero}, we find that $C_k = \emptyset$.
	\end{proof}
	
	{\bf Claim 8. } For $l \in [1, m]$ and $c \in \mathcal A$ such that $\theta_{c}(A_l) \neq \emptyset$, there exists $k \in [1, n]$ such that $a_k = c$. Furthermore, $A_l \in \mathcal B_{\text{reg}}$.

	\begin{proof}
		As in previous claims, we have that \[s_{c, \theta_c(A_l)}^{\ast} x s_{c, \theta_c(A_l)} = s_{c, \theta_c(A_l)}^{\ast} a s_{c, \theta_c(A_l)}\]
		
		We expand the left-hand side to be equal to \[s^{\ast}_{c, \theta_c(A_l)} \left( \sum_{j=1}^m \lambda_j p_{A_j} \right) s_{c, \theta_c(A_l)} = \sum_{j=1}^m \lambda_j s^{\ast}_{c,\theta_c(A_l)} s_{c, \theta_c(A_l)} p_{\theta_c(A_j)} = \lambda_l p_{\theta_{c}(A_l)} \neq 0\] using the pairwise disjointness of $A_j$.
		
		This means that $s_{c, \theta_c(A_l)}^{\ast} a s_{c, \theta_c(A_l)} \neq 0$ and hence there must be $k \in [1, n]$ such that $a_k = c$. This tells us that $\Delta_{A_l} \subseteq \{a_1, \ldots, a_n\}$ and hence, we find that $|\Delta_{A_l}| < \infty$. Now since $x \in \mathrm{Ann}_S(I)^{\perp}$, we have that $A_l \cap B = \emptyset$ for every $B \in \mathcal B_{\text{sink}}$. Namely, this says that for any $A \subseteq A_l$ we have that $A \cap B \subseteq A_l \cap B \notin \mathcal B_{\text{sink}}$ so $|\Delta_{A}| > 0$ for all $A \subseteq A_l$ and hence $A_l \in \mathcal B_{\text{reg}}$.
	\end{proof}
	{\bf Claim 9. } We have that $x = a$ and, in particular, Condition (3) of Theorem 5.1 holds.
	
	\begin{proof}
	By Claim 8, we have that all $A_l$ are regular so we can write \[x = \sum_{i=1}^m \lambda_i p_{A_i} = \sum_{i=1}^n \lambda_i \sum_{c \in \Delta_{A_i}} s_{c, \theta_c(A_i)}s_{c, \theta_c(A_i)}^{\ast}\] Now applying Claim 6 and using the fact that all $c \in \Delta_{A_j}$ appear as some $a_k$, we get
	\[\sum_{i=1}^m \lambda_i \sum_{c \in \Delta_{A_i}} s_{c, \theta_c(A_i)}s_{c, \theta_c(A_i)}^{\ast} = \sum_{i=1}^n \sum_{c \in \Delta_{A_i}} s_{c, \theta_c(A_i)} \left( \sum_{\substack{j \in [1, n] \colon a_j = c \\ \theta_c(A_i) \cap B_j \neq \emptyset}} \mu_j p_{\theta_c(A_i) \cap B_j}\right) s^{\ast}_{c, \theta_c(A_i)}\]
	\[=\sum_{i=1}^m\sum_{c \in \Delta_{A_i}} \sum_{\substack{j \in [1, n] \colon a_j = c}} \mu_j s_{c, \theta_c(A_i) \cap B_j}s^{\ast}_{c, \theta_c(A_i) \cap B_j}\]
	\[=\sum_{i=1}^m \sum_{\substack{j \in [1, n] \colon a_j \in \Delta_{A_i}}} \mu_j s_{a_j, \theta_{a_j}(A_i) \cap B_j}s^{\ast}_{a_j, \theta_c(A_i) \cap B_j}\]
	\[=\sum_{i=1}^m \sum_{j=1}^n \mu_j s_{a_j, \theta_{a_j}(A_i) \cap B_j}s^{\ast}_{a_j, \theta_c(A_i) \cap B_j}\]
	
	On the other hand, by Claim 7 we have that 
	\[a = \sum_{j=1}^n \mu_j s_{a_j, B_j} s_{a_j, B_j}^{\ast} = \sum_{j=1}^n \mu_j s_{a_j, B_j} \left( \sum_{i=1}^m p_{\theta_{a_j}(A_i) \cap B_j} \right) s_{a_j, B_j}^{\ast}\]
	\[=\sum_{i=1}^m \sum_{j=1}^n \mu_j s_{a_j, \theta_{a_j}(A_i) \cap B_j}s^{\ast}_{a_j, \theta_c(A_i) \cap B_j}\]
	so $x = a$ and in Condition (3) holds.
	\end{proof}

	It remains to prove Condition (4). Let $x \in \mathrm{Ann}_S(I) \cap \mathrm{Ann}_S(I)^{\perp}$. By Claim 1, $x = \sum_{j=1}^n \lambda_j p_{A_j}$ where $A_j \in \mathcal B_{\text{sink}}$ where $A_j$ are pairwise disjoint. Let $k \in [1, n]$ be arbitrary. Since $x \in \mathrm{Ann}_S(I)^{\bot}$ and each $p_{A_k} \in \mathrm{Ann}_S(I)$, we get \[0 = p_{A_k}x = p_{A_k}\sum_{j=1}^n \lambda_j p_{A_j} = \lambda_k p_{A_k}\] Thus, we get that $A_k = \emptyset$ for all $k \in [1, n]$ and hence $x = 0$, proving Condition (4).
	
	Finally, note that $S$ contains all generators of the form $p_A$ for $A \in \mathcal B$, $I$ contains all generators of the form $s_{a, A}$ for $a \in \mathcal L$ and $A \in \mathcal I_{a}$, and $J$ contains all generators of the form $s_{a, A}^{\ast}$ for $a \in \mathcal L$ and $A \in \mathcal I_a$, so $L_R(\mathcal B, \mathcal L, \theta, \mathcal I)$ is generated as a ring by $S$, $I$, and $J$. Hence, by Theorem~\ref{ssystem}, $L_R(\mathcal B, \mathcal L, \theta, \mathcal I)$ is graded isomorphic to $\mathcal O_{(J, I, \psi)}$.
\end{proof}
	
	We now obtain a graded uniqueness theorem for generalized Boolean dynamical system algebras which is the algebraic analog of \cite[Corollary~6.2]{CARLSEN2020124037}.
	
	\begin{theorem}\label{gradednonrel}
		Let $(\mathcal B, \mathcal L, \theta, \mathcal I)$ be a generalized Boolean dynamical system and $A$ a $\mathbb Z$-graded $R$-algebra. If $f: L_R(\mathcal B, \mathcal L, \theta, \mathcal I) \rightarrow A$ is a $\mathbb Z$-graded homomorphism of $R$-algebras, then $f$ is injective if and only if $f(rp_A) \neq 0$ for all $\emptyset \neq A \in \mathcal B$ and $0 \neq r \in R$.
	\end{theorem}
	\begin{proof}
		After realizing the algebra as a Cuntz-Pimsner algebra in Theorem~\ref{cuntzrealize} and proving that $rp_A \neq 0$ for $0 \neq r \in R, \emptyset \neq A \in \mathcal B$ using Corollary~\ref{corollary:nonzero}, the proof is the same as \cite[Corollary~5.5]{Boava2021LeavittPA}.
	\end{proof}
	
	We now use the graded uniqueness theorem for generalized Boolean dynamical systems along with the isomorphism from Theorem~\ref{theorem:relativeiso} to prove a graded uniqueness theorem for relative generalized Boolean dynamical systems that is the algebraic analog of \cite[Theorem~6.1]{CARLSEN2020124037}.
	
	\begin{theorem}
		Let $(\mathcal B, \mathcal L, \theta, \mathcal I, \mathcal J)$ be a generalized relative Boolean dynamical system and $A$ a $\mathbb Z$-graded $R$-algebra. If $f: L_R(\mathcal B, \mathcal L, \theta, \mathcal I, \mathcal J) \rightarrow A$ is a $\mathbb Z$-graded homomorphism of $R$-algebras, then $f$ is injective if and only if $f(rp_A) \neq 0$ for all $\emptyset \neq A \in \mathcal B$ and $f(rq_A) \neq 0$ for all $A \in \mathcal B_{\text{reg}} \setminus \mathcal J$ with $0 \neq r \in R$.
	\end{theorem}	
	\begin{proof}
		We show that the condition is sufficient, as the necessity of the condition follows easily from Corollary~\ref{corollary:nonzerorel}. Again we will use the isomorphism $\psi: L_R(\tilde{\mathcal B}, \mathcal L, \tilde{\theta}, \tilde{\mathcal I}) \rightarrow L_R(\mathcal B, \mathcal L, \theta, \mathcal I, \mathcal J)$ from Theorem~\ref{theorem:relativeiso} that takes $\psi(p_{(A, [B]_{\mathcal J})}) = p_A + q_C - q_D$ for $A, B \in \mathcal B$. For ease, we repeat the conditions for our elements $C, D$: \[C, D \in \mathcal B_{\text{reg}}\] \[[A]_{\mathcal B_{\text{reg}}} = [B]_{\mathcal B _{\text{reg}}}\] \[A \cup C = B \cup D \text{ and } A \cap C = \emptyset = B \cap D\] 
		
		Note that the isomorphism $\psi$ is $\mathbb Z$-graded. Hence, $f \circ \psi: L_R(\tilde{\mathcal B}, \mathcal L, \tilde{\theta}, \tilde{\mathcal I}) \rightarrow A$ is a $\mathbb Z$-graded morphism of $R$-algebras as well. Thus, from Theorem~\ref{gradednonrel} we get that if $f(\psi(rp_{(A, [B]_{\mathcal J})})) \neq 0$ for all $A$ and $[B]_{\mathcal J}$ at least one non-empty and $r \neq 0$, then $f \circ \psi$ is injective. Because $\psi$ is an isomorphism, this will also prove that $f$ is injective.
		
		Note that $f(\psi(rp_{(A, [B]_{\mathcal J})})) = f(r(p_A+ q_C - q_D))$. Our goal now is to prove that if $f(r(p_A + q_C - q_D)) = 0$, then there is some $\emptyset \neq Z_1 \in \mathcal B_{\text{reg}} \setminus \mathcal J$ that satisfies $f(rq_{Z_1}) = 0$ or there is some $\emptyset \neq Z_2 \in \mathcal B$ that satisfies  $f(rp_{Z_2}) = 0$ for some $0 \neq r \in R$.
		
		If $A = \emptyset$, then we must have that $B \in \mathcal B \setminus \mathcal J$ because at least one of $A$ or $[B]_{\mathcal J}$ must be non-empty. Because $[A]_{\mathcal B_{\text{reg}}} = [B]_{\mathcal B_{\text{reg}}}$ we get that $B \in \mathcal B_{\text{reg}}$.  It's easy to show that we can choose $C = B$ and $D = \emptyset$ so we get that $f(r(p_A + q_C - q_D)) = f(rq_B) = 0$ with $B \in \mathcal B_{\text{reg}} \setminus \mathcal J$ and and $r \neq 0$, so this case is done.
		
		Now assume that $A \neq \emptyset$. We multiply $f(r(p_A + q_C - q_D)) = 0$ by $f(p_A)$ and use the fact that $A \cap C = \emptyset$ and the calculations in Lemma~\ref{lemma:qcalc} to deduce that $f(r(p_A - q_{D \cap A})) = 0$. If $D \cap A \subsetneq A$, then multiply by $f(p_{A \setminus (D \cap A)})$ to get that $f(rp_{A \setminus (D \cap A)}) = 0$ for $A \setminus (D \cap A) \neq \emptyset$ and $r \neq 0$. Otherwise, $D \cap A = A \Rightarrow A \subseteq D$. Because $D \in \mathcal B_{\text{reg}}$, we find that $A \in \mathcal B_{\text{reg}}$ and hence so is $B$ because $[A]_{\mathcal B_{\text{reg}}} = [B]_{\mathcal B_{\text{reg}}}$. Knowing that $A$ and $B$ are regular, we redefine $C \coloneqq B \setminus A$ and $D \coloneqq A \setminus B$.
		
		If $C \notin \mathcal J$, we multiply $f(r(p_A + q_C - q_D)) = 0$ by $f(p_C)$ and note that $D \cap C = \emptyset$ and $A \cap C = \emptyset$ so we get that $f(rq_C) = 0$ for $C \in \mathcal B_{\text{reg}} \setminus \mathcal J$.
		
		Otherwise, assume $C \in \mathcal J$ so we get that $q_C = 0$ and thus \[0 = f(r(p_A - q_A)) = f\left(r\sum_{a \in \Delta_A} s_{a, \theta_{a}(A)}s^{\ast}_{a, \theta_{a}(A)}\right)\] Choose some $a \in \Delta_A$ and conjugate our expression by $f(s_{a, \theta_{a}(A)}^{\ast})$ and $f(s_{a, \theta_{a}(A)})$ to get $f(rp_{\theta_{a}(A)}) = 0$ where $\theta_{a}(A) \neq \emptyset$ because $a \in \Delta_A$, so we are done.
	\end{proof}

	\section{Graded Ideal Structure} \label{section:gradedideal}
	Throughout, let $(\mathcal B, \mathcal L, \theta, \mathcal I, \mathcal J)$ be a relative generalized Boolean dynamical system. As an application of the graded uniqueness theorem, we describe the graded ideal structure for relative generalized Boolean dynamical systems. 	The proofs in this section are the same symbolic manipulations as in \cite[Section~7]{CARLSEN2020124037} so we omit them for the sake of brevity. All following definitions are found in \cite[Section~7]{CARLSEN2020124037}. 
	\begin{definition}
			For any ideal $H \subseteq \mathcal B$, we say $H$ is \textit{hereditary} if it is closed under our morphisms $\theta_{\alpha}(\cdot)$. We say a hereditary ideal $H \subseteq \mathcal B$ is $\mathcal J$-saturated if, for any $A \in \mathcal J$ such that $\{\theta_{\alpha}(A) \colon \alpha \in \mathcal L\} \subseteq H$, we have that $A \in H$.
	\end{definition}	

	\begin{definition}
		For any hereditary $\mathcal J$-saturated ideal $H$ of $\mathcal B$ there is an equivalence relation on $\mathcal B$ defined as $A \sim_{H} B \Leftrightarrow \exists (C \in H)(A \cup C = B \cup C)$. Define $\mathcal B/H \coloneqq \{[A]_H \colon A \in \mathcal B\}$ which is another generalized Boolean algebra with the natural operations. One can also define the sets $\mathcal I/H = \{\mathcal I_{\alpha}/H\}_{\alpha \in \mathcal L}$ and functions $\theta/H = \{\theta_{\alpha}/H\}_{\alpha \in \mathcal L}$ in a similar way. The quadruplet $(\mathcal B/H, \mathcal L, \theta/H, \mathcal I/H)$ forms a generalized Boolean dynamical system. 
		
		Define a generalized Boolean algebra $\mathcal B_{H} \coloneqq \{A \in \mathcal B \colon [A]_H \in (\mathcal{B}/H)_{\text{reg}}\}$. For any ideal $S \subseteq \mathcal B_H$ such that $H \cup \mathcal J \subseteq S$, we define the set \[S_{H} = \{p_A - \sum_{\alpha \in \Delta^{(\mathcal B/H, \mathcal L, \mathcal \theta/H, \mathcal I/H)}_{[A]_H}} s_{\alpha, \theta_{\alpha}(A)}s_{\alpha, (A)}^{\ast} \colon A \in S\}\] Denote $I_{(H, S)}$ to be the ideal generated by $S_H$ in $L_R(\mathcal B, \mathcal L, \theta, \mathcal I, \mathcal J)$.
	\end{definition}
	
	\begin{remark}
	Note that $S/H = \{[A]_H \colon A \in S\}$ is an ideal of $(\mathcal B/H)_{\text{reg}}$. Hence, \[(\mathcal{B}/H, \mathcal L, \theta/H, \mathcal I/H, S/H)\] forms a relative generalized Boolean dynamical system.
	\end{remark}
	
	We now list the analogs of Lemma~7.2, Proposition~7.3, and Theorem~7.4 in \cite{CARLSEN2020124037} for $C^{\ast}$-algebras. After the development of the graded uniqueness theorems, the proofs are the same.
	
	\begin{lemma}[{\cite[Lemma~7.2]{CARLSEN2020124037}}] \label{lemma:idealhered}
		Let $I \subseteq L_R(\mathcal B, \mathcal L, \theta, \mathcal I, \mathcal J)$ be an ideal, then the set $H_I = \{A \in \mathcal B \colon p_A \in H\}$ is a hereditary $\mathcal J$-saturated set. The set $S_I = \{A \in \mathcal B_H \colon p_A - \sum_{\alpha \in \Delta_{[A]_{H_I}}} s_{\alpha, \theta_{\alpha}(A)}s^{\ast}_{\alpha, \theta_{\alpha}(A)} \in I\}$ is an ideal of $\mathcal B_H$ containing $H_I$ and $\mathcal J$.
	\end{lemma}

	\begin{theorem}[{\cite[Proposition~7.3]{CARLSEN2020124037}}]
		Let $I \subseteq L_R(\mathcal B, \mathcal L, \theta, \mathcal I, \mathcal J)$ be an ideal. Let $H_I$ and $S_I$ be defined as before. There is a surjective graded homomorphism \[\varphi_I: L_R(\mathcal B/H_I, \mathcal L, \theta/H_I, \mathcal I/H_I, S_I/H) \rightarrow L_R(\mathcal B, \mathcal L, \theta, \mathcal I, \mathcal J)/I\] \[\varphi([A]_{H_I}) = p_A + I \text{ and } \varphi(s_{a, [A]_{H_I}}) = s_{a, A} + I\]
		
		Furthermore, the following are equivalent
		\begin{enumerate}
			\item $I$ is graded
			\item The map $\varphi_I$ is an isomorphism
			\item $I = I_{(H_I, S_I)}$
		\end{enumerate}
	\end{theorem}

	The set of all pairs $(H, S)$ where $H$ is a hereditary $\mathcal J$-saturated ideal of $\mathcal B$ and $S$ is an ideal of $\mathcal B_H$ with $H \cup \mathcal J \subseteq S$ is a lattice with respect to the order $(H_1, S_1) \subseteq (H_2, S_2) \Leftrightarrow H_1 \subseteq H_2 \text{ and } S_1 \subseteq S_2$. The set of graded ideals is a lattice with respect to set-inclusion.

	\begin{theorem}[{\cite[Theorem~7.4]{CARLSEN2020124037}}]
		Let $(\mathcal B, \mathcal L, \theta, \mathcal I, \mathcal J)$ be a relative generalized Boolean dynamical system. Then the map $(H, S) \mapsto I_{(H, S)}$ is a lattice isomorphism between the lattice of all pairs $(H, S)$ where $H$ is a hereditary $\mathcal J$-saturated ideal of $\mathcal B$ and $S$ is an ideal of $\mathcal B_H$ with $H \cup \mathcal J \subseteq S$ and the lattice of all graded ideals of $L_R(\mathcal B, \mathcal L, \theta, \mathcal I, \mathcal J)$.
	\end{theorem}

	\section{Desingularization} \label{section:desing}
	In this section, we will show that all generalized Boolean dynamical system algebras with countable alphabets $\mathcal L$ are Morita equivalent to the algebra of a generalized Boolean dynamical system without singular sets. We will prove it in the case of countably infinite $\mathcal L$ as the case of finite $\mathcal L$ is essentially the same.
	
	Let $(\mathcal B, \mathcal L, \theta, \mathcal I)$ be a generalized Boolean dynamical system where $\mathcal L = \{a_1, a_2, \ldots\}$ is some countably infinite enumeration. We will first construct a new generalized Boolean dynamical system $(\mathcal B^F, \mathcal L^F, \theta^F, \mathcal I^F)$. 

	\begin{definition}
		For $U \in \mathcal B$, we say that $U$ has \textit{no sinks} if $U$ has empty intersection with every set in $\mathcal B_{\text{sink}}$.
	\end{definition}
	
	\begin{remark}
		The emptyset has no sinks vacuously. Furthermore, if $U$ has no sinks, then every $V \subseteq U$ has no sinks.
	\end{remark}	

	We start by defining a sequence of ideals $X_i \subseteq \mathcal B$ for $i \geq 0$. Define $X_0 = \{\emptyset\}$. For $i \geq 1$, define $X_i = \{U \in \mathcal B \colon U \text{ has no sinks and } \Delta_U \subseteq \{a_1, \ldots, a_{i-1}\}\}$. Note that $X_1 = \{\emptyset\}$ because the only element $U \in \mathcal B$ with no sinks and $\Delta_U = \emptyset$ is the emptyset.
	
	\begin{lemma} The following are true
		\begin{enumerate}
			\item $X_i \subseteq X_{i+1}$
			\item $X_i$ is an ideal of $\mathcal B$
		\end{enumerate}
	\end{lemma}
	\begin{proof} We prove each statement separately.
		\begin{enumerate}
			\item Note that $\emptyset \in X_i$ for all $i \geq 0$. Thus, $X_0 \subseteq X_1$. Now let $i \geq 1$ and take $U \in X_i$. Then we know that $U$ has no sinks and that $\Delta_U \subseteq \{a_1, \ldots, a_{i-1}\}\subseteq \{a_1, \ldots, a_i\}$, hence $U \in X_{i+1}$ as well.
			
			\item $X_0$ is clearly an ideal. To show that $X_i$ is an ideal for $i \geq 1$, let $U \in X_i$ and let $V \subseteq U$. By our remark, $V$ has no sinks. Furthermore, $\Delta_V \subseteq \Delta_U \subseteq \{a_1, \ldots, a_{i-1}\}$, hence $V \in X_i$.
		\end{enumerate}
	\end{proof}

	For our above ideals $X_i$, we define $\mathcal B_i \coloneqq \mathcal B/X_i$. For $A \in \mathcal B$, we refer to the equivalence class $[A] \in \mathcal B/X_i$ as $[A]_i$. We define our new generalized Boolean algebra $\mathcal B^F$ as the set of all formal finite unions of sets from $\{B_i\}_{i=0}^{\infty}$, denoted by \[\mathcal B^F \coloneqq \bigoplus_{i=0}^\infty \mathcal B_i\]
	
	\begin{remark}
		Any $B \in \mathcal B^F$ can be represented by $\bigsqcup_{i=0}^{\infty} [B_i]_i$ where $B_i \in \mathcal B$ and all but a finite number of $B_i$ are the emptyset. For $A \in \mathcal B$, we abuse notation and refer to the copy of $A$ viewed in the $i$'th level of $\mathcal B^F$ as $[A]_i$.
	\end{remark}
	
	We define our language $\mathcal L^F$ to be \[\mathcal L^F \coloneqq \mathcal L \cup \{b_1, b_2, \ldots\}\] For $i = 1, 2, \ldots$ and $b_i \in \mathcal L^F$, we define \[\theta^F_{b_i}([A]_{i-1}) = [A]_i \text{ for } A \in \mathcal B\] Namely, $\theta^F_{b_i}$ is defined as the quotient map from $\mathcal B_i \rightarrow \mathcal B_{i+1}$ and $0$ everywhere else. For $a_i \in \mathcal L$, we define  \[\theta^F_{a_i}([A]_i) = [\theta_a(A)]_0\] where the map is $0$ everywhere else. 
	
	If the maps are well-defined, then they are clearly morphisms $\mathcal B^F \rightarrow \mathcal B^F$. The maps $\theta^F_{b_i}$ are well-defined because they are just quotient maps. To prove that $\theta_{a_i}^F$ is well-defined, let $A, B \in \mathcal B$ be such that $[A]_i = [B]_i$ which implies that $A \cup U = B \cup U$ for $U \in X_i$. Note that $\theta_{a_i}(A \cup U) = \theta_{a_i}(A) \cup \theta_{a_i}(U)$. Because $\Delta_U \subseteq \{a_1, \ldots, a_{i-1}\}$, we know that $\theta_{a_i}(U) = \emptyset$. Hence, $\theta_{a_i}(A \cup U) = \theta_{a_i}(A)$ and $\theta_{a_i}(B \cup U) = \theta_{a_i}(B)$, so we get that $\theta_{a_i}(A) = \theta_{a_i}(B)$ and we are done.
	
	To define our ideals $\mathcal I^F_{c}$ for $c \in \mathcal L_F$, recall that $\mathcal I^F_{c}$ must contain $\mathcal F^F_{c} = \{A \in \mathcal B^F \colon \exists B \in \mathcal B^F, A \subseteq \theta^F_{c}(B)\}$. It's not hard to see that $\mathcal F^F_{b_i} = \mathcal B_i$ and $\mathcal F_{a_i}^F = [\mathcal F_{a_i}]_0$ where the notation $[\mathcal F_{a_i}]_0$ denotes $\mathcal F_{a_i}$ viewed inside the copy $\mathcal B_0 = \mathcal B$. With this in mind, we define our ideals as \[\mathcal I^F_{a_i} = [\mathcal I_{a_i}]_0 \text{ and } \mathcal I^F_{b_i} = \mathcal B_i \text{ for } i \geq 1\]

	It's not hard to prove from this that $(\mathcal B^F, \mathcal L^F, \theta^F, \mathcal I^F)$ forms a generalized Boolean dynamical system.
	
	\begin{theorem}
		$(\mathcal B^F)_{\text{reg}} = \mathcal B^F$.
	\end{theorem}

	\begin{proof}
		Let $\emptyset \neq B = \bigsqcup_{i=0}^{\infty} [B_i]_i \in \mathcal B^F$ where all but a finite number of $B_i$ are the emptyset.
		
		We first show that $|\Delta^{\mathcal B^F}_B| < \infty$. It's easy to calculate that $\theta_{a_i}([B_i]_i) = [\theta_{a_i}(B_i)]_0$ and $\theta_{b_{i+1}}([B_i]_i) = [B_i]_{i+1}$. Each of these non-empty only if $B_i \neq \emptyset$. Because $B_i$ is not the emptyset for only a finite number of $B_i$, we have that $|\Delta_B| < \infty$.
		
		We now show that $|\Delta^{\mathcal B^F}_B| > 0$. First note that because $B \neq \emptyset$, at least one of the $[B_i]_i \neq \emptyset \Rightarrow B_i \notin X_i$. If $\theta_{b_{i+1}}([B_i]_i) \neq \emptyset$, then $\theta_{b_{i+1}(B)} \neq \emptyset$ so $b_i \in \Delta^{\mathcal B^F}_B$. Now, assume that $\theta_{b_{i+1}}([B_i]_i) = [B_i]_{i+1} = \emptyset \Rightarrow B_i \in X_{i+1}$. Thus, we find that $B_i \in X_{i+1} \setminus X_i$. Because $B_i \in X_{i+1}$, we know that $B_i$ has no sinks and $\Delta^{\mathcal B}_{B_i} \subseteq \{a_1, \ldots, a_i\}$. Because $B_i \notin X_i$ and $B_i$ has no sinks, this means that $\Delta^{\mathcal B}_{B_i} \not \subseteq \{a_1, \ldots, a_{i-1}\}$. Hence, we conclude that $a_i \in \Delta^{\mathcal B}_{B_i}$ hence $\theta_{a_i}(B) = \theta_{a_i}([B_i]_i) = [\theta_{a_i}(B_i)]_0 \neq \emptyset$ so $a_i \in \Delta^{\mathcal B^F}_B$.
	\end{proof}

	Our main theorem is stated below.
	\begin{theorem}
		Let $(\mathcal B, \mathcal L, \theta, \mathcal I)$ be a generalized Boolean dynamical system with countably infinite alphabet $\mathcal L$. Then, $L_R(\mathcal B, \mathcal L, \theta, \mathcal I)$ and $L_R(\mathcal B^F, \mathcal L^F, \theta^F, \mathcal I^F)$ are Morita equivalent.
	\end{theorem}

	\begin{proof}
	We will show that $S_{(\mathcal B, \mathcal L, \theta, \mathcal I)} \subseteq S_{(\mathcal B^F, \mathcal L^F, \theta^F, \mathcal I^F)}$ and satisfies the conditions of \cite[Theorem~5.7]{zhang2025moritaequivalencesubringsapplications}. The proof is similar to \cite[Section~8]{zhang2025moritaequivalencesubringsapplications}, so we omit parts which are essentially the same. For ease of notation, let $S_1 \coloneqq S_{(\mathcal B, \mathcal L, \theta, \mathcal I)}$ and $S_2 \coloneqq S_{(\mathcal B^F, \mathcal L^F, \theta^F, \mathcal I^F)}$.

	Define a map \[h: \mathbb F[\mathcal L] \rightarrow \mathcal F[\mathcal L^F] \text{ that takes }a_i \mapsto b_1\ldots b_i a_i\] and define a map $S_1 \hookrightarrow S_2$ that takes $(\alpha, A, \beta) \mapsto (h(\alpha), [A]_0, h(\beta))$.

	\begin{lemma}
		The map $S_1 \hookrightarrow S_2$ that takes $(\alpha, A, \beta) \mapsto (h(\alpha), [A]_0, h(\beta))$ is a well defined injective morphism between inverse semigroups.
	\end{lemma}
	\begin{proof}
		Same as \cite[Theorem~8.8]{zhang2025moritaequivalencesubringsapplications}.
	\end{proof}

	\begin{lemma} \label{singlecalc} The following calculations are used throughout. Verifying them is easy so we omit the proof.
		\begin{enumerate}
		\item $\theta^F_{b_{i+1}b_{i+2}\ldots b_n}([A]_i) = [A]_n$ for $n \geq i$ and $\theta^F_{b_j}([A]_i) = \emptyset$ for $j \neq i+1$
		\item $\theta^F_{h(\alpha')}([A]_0)= [\theta_{\alpha'}(A)]_0$ for $\alpha' \in \mathcal L^{\ast}$
		\item $\mathcal I^F_{h(\alpha')b_1\ldots b_n} = [\mathcal I_{\alpha'}]_n$ for any $n \in [0, \infty)$ and $\omega \neq \alpha' \in \mathcal L^{\ast}$
		\item If $A \in \mathcal I_{\alpha'}$, then $[A]_n \in \mathcal I^F_{h(\alpha')b_1\ldots b_n}$
		\item If $\alpha_1 = b_1$ and $\mathcal I^F_{\alpha} \neq \{\emptyset\}$ for some $\alpha \in (\mathcal L^F)^{\ast}$, then $\alpha = h(\alpha')b_1\ldots b_n$ for some $\alpha' \in \mathcal L^{\ast}$ and $n \geq 0$
		\item $\mathcal I^F_{b_i} \cap \mathcal I^F_{b_j} = \{\emptyset\}$ for any $i \neq j$ and $\mathcal I^F_{b_i} \cap \mathcal I^F_{a_j} = \{\emptyset\}$ for any $i$ and $j$
		\item If a non-empty word $\alpha \in \mathcal (L^F)^{\ast}$ satisfies that $\alpha_1 = b_1$ and $\mathcal I^F_{\alpha} \neq \{\emptyset\}$, then $\alpha = h(\alpha')b_1\ldots b_n$ for $\alpha' \in \mathcal L^{\ast}$ and some $n \geq 0$
		\end{enumerate}
	\end{lemma}
	
	We now begin building lemmas to prove the conditions of \cite[Theorem~5.7]{zhang2025moritaequivalencesubringsapplications}. We first characterize the elements of $S_2$ that are also elements of $S_1$.
	
	\begin{lemma} \label{ins1}
		Let $(\alpha, A, \beta) \in S_2$. Then $(\alpha, A, \beta) \in S_1$ if and only if $A \in \mathcal B_0$ and, for $\alpha$ and $\beta$, they individually are either empty or their first characters are $b_1$.
	\end{lemma}

	\begin{proof}
	
		Same as \cite[Lemma~8.9]{zhang2025moritaequivalencesubringsapplications}.
		
	\end{proof}

	The below corollary follows immediately from restricting our statement to the idempotents $E_1 \subseteq E_2$.

	\begin{corollary} \label{ins1cor}
		Let $(\alpha, A, \alpha) \in E_2$. Then $(\alpha, A, \alpha) \in E_1$ if and only if $A \in \mathcal B_0$ and either $\alpha = \omega$ or $\alpha_1 = b_1$.
	\end{corollary}

	Recall that, for an inverse semigroup $S$, one can define an order on $S$ extending the order on the idempotents by $s_1 \leq s_2 \Leftrightarrow s_1 = es_2$ for some idempotent $e \in S$. Let $s^+$ denote all the elements of $S$ that are greater than $s$. The below lemma gives a sufficient condition for $s^+ \cap S_1 \neq \emptyset$. The result is used to simplify several subsequent proofs.
	
	\begin{lemma} \label{abovecase}
		Let $0 \neq (\alpha, A, \beta) \in S_2$. If $\alpha$ and $\beta$ are not the empty word and $\alpha_1 = b_1$ and $\beta_1 = b_1$, then $(\alpha, A, \beta)^+ \cap S_1 \neq \emptyset$.
	\end{lemma}

	\begin{proof}		
		Now assume that $\alpha_1 = b_1$ and $\beta_1 = b_1$. By Lemma~\ref{singlecalc} (5), we find that $\alpha = h(\alpha')b_1 \ldots b_n$ and $\beta = h(\beta')b_1\ldots b_m$ and by Lemma~\ref{singlecalc} (6) we see that $m = n$. 
		
		If $n = 0$, then we must have that $|\alpha'|, |\beta'| > 0$ and furthermore $\alpha$ and $\beta$ end with some $a_{\alpha}$ and $b_{\beta}$ respectively. By Lemma~\ref{singlecalc} (3), we find that $A \in \mathcal B_0$ and thus by by Lemma~\ref{ins1}, we find that $(\alpha, A, \beta) \in S_1$.
		
		Otherwise, we have that $n > 0$ and hence $A = [A]_n$ for some $A \in \mathcal B$ (abusing notation for $A$). Note that $[A]_n \subseteq [B_{\alpha}]_n$ for some $B_{\alpha} \in \mathcal I_{\alpha'}$ and similarly $[A]_n \subseteq [B_{\beta}]_n$ for some $B_{\beta} \in  \mathcal I_{\beta'}$ using Lemma~\ref{singlecalc} (3). Hence, there exists some $U \in X_n$ such that $A \subseteq B_{\alpha} \cup U$ and $A \subseteq B_{\beta} \cup U$. Note that $[A]_n \neq \emptyset$ so $A \not\subseteq U\Rightarrow A \setminus U \neq \emptyset$. Note furthermore that $A \setminus U \subseteq B_{\alpha} \cap B_{\beta}$, so we get that $A \setminus U \in \mathcal I_{\alpha'} \cap \mathcal I_{\beta'} \Rightarrow [A \setminus U]_0 \in \mathcal I^F_{h(\alpha')} \cap \mathcal I^F_{h(\beta')}$. Hence, we get that $(h(\alpha'), [A \setminus U] _0, h(\beta'))$ is a valid element of $S_2$. By Lemma~\ref{ins1}, we get that $(h(\alpha'), [A \setminus U] _0, h(\beta')) \in S_1$.
		
		Finally, note that $(\alpha, [A]_n, \alpha)(h(\alpha'), [A \setminus U] _0, h(\beta')) = (\alpha, [A \setminus U]_n, \beta) =  (\alpha, [A]_n, \beta)$ because $[A \setminus U]_n = [A]_n$ as $U \in X_n$. Hence, $ (\alpha, A, \beta) \leq (h(\alpha'), [A \setminus U] _0, h(\beta'))$, so we are done.
	\end{proof}

	\begin{corollary} \label{abovecasecorol}
		Let $0 \neq (\alpha, A, \alpha) \in E_2$. Then $(\alpha, A, \alpha)^+ \cap E_1 \neq \emptyset$ if and only if one of the following holds
		\begin{enumerate}
		\item $\alpha = \omega$ and $A \in \mathcal B_0$
		\item $\alpha_1 = b_1$
		\end{enumerate}
	\end{corollary}
	\begin{proof}
		(1) is sufficient because, in this case, we deduce that $(\alpha, A, \alpha) \in E_1$ by Corollary~\ref{ins1cor}. (2) is sufficient from an application of Lemma~\ref{abovecase}.
		
		To see that at least one of the conditions is necessary, let $(\beta, B, \beta) \in E_1$ such that $(\alpha, A, \alpha) \leq (\beta, B, \beta)$. This means that $\alpha = \beta \gamma$ and $A \subseteq \theta^F_{\gamma}(B)$. Using Corollary~\ref{ins1cor} we know that $B \in \mathcal B_0$ and ether $\beta = \omega$ or $\beta_1 = b_1$.
		
		If $\beta_1 = b_1$, then $\alpha_1 = \beta_1 = b_1$ which is (2), so we are done. If $\beta = \omega$, then there are two additional cases. If $\gamma = \omega$, then $\alpha = \omega$ and $A \subseteq B \in \mathcal B_0$ so $A \in \mathcal B_0$ which is (1), so are done. If $\gamma \neq \omega$, then $\alpha_1 = \gamma_1$. If $\gamma_1 \neq b_1$, then $\theta^F_{\gamma}(B) = \emptyset$ as $B \in \mathcal B_0$. Hence, $\alpha_1 = \gamma_1 = b_1$ which again is (1), so we are done.
		
	\end{proof}

	We now check the 4 conditions of \cite[Theorem~5.7]{zhang2025moritaequivalencesubringsapplications}.
	
	\begin{theorem}
		$E_1 \subseteq E_2$ preserves finite covers.
	\end{theorem}

	\begin{proof}
		We use \cite[Lemma~3.18]{zhang2025partialactionsgeneralizedboolean}. Let $(\alpha, A, \alpha) \in E_2$ with $\alpha \in \mathcal (L^F)^{\ast}$ and denote $A = \bigsqcup_{i=0}^{\infty} [A_i]_i \in \mathcal I^F_{\alpha}$ for $A_i \in \mathcal B$.
			
		If $(\alpha, A, \alpha)^+ \cap E_1 = \emptyset$ then there is nothing to prove. By Corollary~\ref{abovecasecorol}, the remaining cases are $\alpha = \omega$ and $A \in \mathcal B_0$ or $\alpha$ begins with $b_1$. 
		
		If $\alpha = \omega$ and $A \in \mathcal B_0$, then $(\alpha, A, \alpha) \in E_1$ by Corollary~\ref{ins1cor}. If $\alpha$ begins with $b_1$ and ends with some $a_i$, then $A \in \mathcal B_0$ so $(\alpha, A, \alpha) \in E_1$ by Corollary~\ref{ins1cor} again.
		
		Hence, we may assume that $\alpha$ begins with $b_1$ and ends with some $b_n$ with $n \geq 1$ and hence $A = [A]_n$ for some $A \in \mathcal B$. Note that $\alpha = h(\alpha')b_1\ldots b_n$ for some $\alpha' \in \mathcal L^{\ast}$ by Lemma~\ref{singlecalc} (7). Because $[A]_n \in \mathcal I^F_{\alpha} = [\mathcal I_{\alpha'}]_n$ we without loss of generality assume that $A \in \mathcal I_{\alpha'}$. Furthermore, $[A]_n \neq \emptyset$, hence there either exists $k \geq n$ such that $a_k \in \Delta^{\mathcal B}_A$ or $A$ has a sink.
		
		If there is some $k$ such that $a_k \geq n$ and $a_k \in \Delta^{\mathcal B}_A$, we see that \[0 \neq (\alpha b_{n+1} \ldots b_k a_k, [\theta_{a_k}(A)]_0, \alpha b_{n+1} \ldots b_k a_k) \leq (\alpha, [A]_n, \alpha)\] If $A$ has a sink, let $\emptyset \neq B \subseteq A$ such that $\Delta_B^{\mathcal B} = \emptyset$. Because $A \in \mathcal I_{\alpha'}$, so is $B$. Hence, the element $(h(\alpha'), [B]_0, h(\alpha')) \in S_2$ is a valid element. We will show that $(\gamma, C, \gamma) \leq (h(\alpha'), [B]_0, h(\alpha'))$ satisfy that $(\gamma, C, \gamma)(h(\alpha')b_1\ldots b_n, [A]_n, h(\alpha')b_1\ldots b_n) \neq 0$, which will prove our theorem.
		
		Note that if $(\gamma, C, \gamma) \leq (h(\alpha'), [B]_0, h(\alpha'))$, then $\gamma = h(\alpha')\beta$ with $\emptyset \neq C \subseteq \theta^F_{\beta}([B]_0)$. However, because $\Delta^{\mathcal B}_B = \emptyset$, $\beta$ has no elements of the form $a_i$. Hence, $\beta = b_1 b_2 \ldots b_m$ for some $m$ possibly $0$ and $C = [C]_m$ for some $C \in \mathcal B$ with $[C]_m \subseteq [B]_m$.
		
		Because $B \subseteq A$, we calculate that \[(h(\alpha')b_1\ldots b_m, [C]_m, h(\alpha')b_1\ldots b_m)(h(\alpha')b_1\ldots b_n, [A]_n, h(\alpha')b_1\ldots b_n) \neq 0\] so we are done.
		
	\end{proof}

	\begin{theorem}
		$E_1 \subseteq E_2$ is tight.
	\end{theorem}

	\begin{proof}
		We use \cite[Lemma~3.23]{zhang2025moritaequivalencesubringsapplications}. Let $(\alpha, A, \alpha) \in E_2$ be arbitrary. We want to prove that there exists some $y \in E_1$ and $\{y_1, \ldots, y_n\} \subseteq E_1$ such that $y_i \wedge (\alpha, A, \alpha) = 0$ and $\{y_1, \ldots, y_n, x\}$ form a cover for $y$.
	
		Because $(\alpha, A, \alpha) \in E_1$ implies that $y = (\alpha, A, \alpha)$ suffices, by Corollary~\ref{abovecasecorol} and the same reasoning as above, we may assume that $\alpha_1 = b_1$ and ends with some $b_n$ with $n \geq 0$. Hence, we let $\alpha = h(\alpha')b_1\ldots b_n$ and let $A = [A]_n$ with $A \in \mathcal I_{\alpha'}$.
		
		Again, we will consider $y = (h(\alpha'), [A]_0, h(\alpha'))$. We define \[y_i = (h(\alpha'a_i), [\theta_{a_i}(A)]_0, h(\alpha'a_i))\] where $y_i = 0$ whenever $\theta_{a_i}(A) = \emptyset$. These are all clearly valid elements and in $E_1$ by Lemma~\ref{ins1cor}.
		
		We will show that $\{y_1, \ldots, y_{n-1}, x\}$ form a cover for $y$. Let $(\beta, B, \beta) \leq (h(\alpha'), [A]_0, h(\alpha')) = y$ so $\beta = h(\alpha')\gamma$ and $B \subseteq \theta^F_{\gamma}([A]_0)$. It suffices to show that there is some $(\beta', B', \beta') \in E_2$ such that $(\beta', B', \beta') \leq (\beta, B, \beta)$ and $\beta'$ has prefix of any of the following: \[h(\alpha'a_1), \ldots, h(\alpha'a_{n-1}), h(\alpha')b_1\ldots b_n\]
		
		Note that because we have $\emptyset \neq B \subseteq \theta^F_{\gamma}([A]_0)$, we have that if $\gamma$ is non-empty then $\gamma_1 = 1$ and $\mathcal I^F_{\gamma} \neq \{\emptyset\}$. Hence, $\gamma = h(\gamma')b_1 \ldots b_m$ for some $\gamma' \in \mathcal L^{\ast}$.
		
		If $\gamma' \neq \omega$, then if $\gamma'_1 = a_i$ for $i \leq n-1$, then $\gamma$ has $h(\alpha'a_i)$ as a prefix. Otherwise, $\gamma$ has $h(\alpha')b_1\ldots b_n$ as a prefix, so we are done.
		
		Otherwise, $\gamma' = \omega$ and hence $\gamma = b_1 \ldots b_m$. By similar techniques as before, we see that at least one of the following is a valid non-zero elements \[(h(\alpha')b_1 \ldots b_m b_{m+1}, \theta^F_{b_{m+1}}(B), h(\alpha')b_1 \ldots b_m b_{m+1})\] \[(h(\alpha')b_1 \ldots b_m a_m, \theta^F_{b_{m+1}}(a_m), h(\alpha')b_1 \ldots b_m a_m)\] Repeating this enough times, we eventually get an element of $E_2$ with our desired prefix, so we are done.
		
		Finally, it's obvious that there is no element $(\alpha, A, \alpha) \in E_2$ that has both prefix of the form $h(\alpha'a_i)$ for some $i \leq n-1$  and $h(\alpha')b_1 \ldots b_n$. Hence, $y_i \wedge x = 0$ for all $i \leq n-1$ and we are done.
	\end{proof}

	\begin{theorem}
		For all $x, y \in E_1$ and $s \in S_2$ there exists $s' \in S_1$ such that $xsy \leq s'$.
	\end{theorem}
	\begin{proof}
		Let $s = (\alpha, A, \beta)$, $x = (\alpha_x, A_x, \alpha_x)$, and $y  = (\alpha_y, A_y, \alpha_y)$. By Lemma~\ref{ins1cor}, we have that $A_x, A_y \in \mathcal B_0$ and $\alpha_x, \beta_x$ are either the empty word or their first characters are $b_1$.
		
		If $xsy = 0$ then we are done, so from now on assume that $xsy \neq 0$ so we write $(\alpha', A', \beta') = xsy$. We will show that $(\alpha', A', \beta')$ satisfies the condition for Lemma~\ref{abovecase}.
		
		We first show that $\alpha'$ is either the empty word or begins with $b_1$, and similarly for $\beta'$. We will prove this only for $\alpha'$ as the proof for $\beta'$ is the same (or derived from taking inverses).
		
		Note that $x(\alpha', A', \beta') = x(xsy) = xsy \neq 0$. Assume for the sake of contradiction that $\alpha'$ is not the empty word and begins with a different character than $b_1$. If $\alpha_x = \omega$, then $(\omega, A_x, \omega)(\alpha', A', \beta') = (\alpha', \theta^F_{\alpha'}(A_x) \cap A', \beta') = 0$ because $\theta^F_{\alpha'_1}(A_x) = \emptyset$ as $A_x \in \mathcal B_0$ and $\alpha'_1 \neq b_1$. If $|\alpha_x| \neq \omega$, then $x(\alpha', A', \beta') = 0$ because $\alpha_x$ begins with $b_1$ and $\alpha'$ does not, so we are done.
		
		We now know that $\alpha', \beta'$ are either the empty word or begin with $b_1$. If both begin with $b_1$, then we are done by Lemma~\ref{abovecase}. Hence, assume that at least one of $\alpha', \beta'$ are $\omega$. We will show that in this case $(\alpha', A', \beta') \in E_1$. By Lemma~\ref{ins1}, it suffices to show that $A \in \mathcal B_0$. Again the proof is the same for $\beta' = \omega$, so we just assume that $\alpha = \omega$.
		
		As before, $x(xsy) = xsy$ so $(\alpha_x, A_x, \alpha_x)(\omega, A', \beta') = (\omega, A', \beta')$. We calculate that \[(\alpha_x, A_x, \alpha_x)(\omega, A', \beta') = (\alpha_x, A_x \cap \theta_{\alpha_x}(A'), \beta'\alpha_x) = (\omega, A', \beta')\] Hence, we get that $A' \subseteq A_x \in \mathcal B_0$, so $A' \in \mathcal B_0$ and we are done.		
	\end{proof}

	\begin{theorem}
		$L_R(S_1) \subseteq L_R(S_2)$ is contained in no proper two sided ideal
	\end{theorem}
	
	\begin{proof}
		Our map on the algebras takes $p_A \mapsto p_{[A]_0}$ and $s_{a_n, A} \mapsto s_{b_1\ldots b_n, [A]_0}$. Let $I$ be any two-sided ideal in $L_R(\mathcal B^F, \mathcal L, \theta^F, \mathcal I^F)$ containing $L_R(\mathcal B, \mathcal L, \theta, \mathcal I)$. We know from Lemma~\ref{lemma:idealhered} that $H_I = \{A \in \mathcal B^F \colon p_A \in I\} \subseteq \mathcal B^F$ is hereditary. Furthermore, it's not hard to see that if $\{p_B\}_{B \in \mathcal B^F} \subseteq I$, then $I = L_R(\mathcal B^F, \mathcal L, \theta^F, \mathcal I^F)$.
		
		We aim to show that $\{p_A \colon A \in \mathcal B^F\} \subseteq I$. Because $I$ contains $L_R(\mathcal B, \mathcal L, \theta, \mathcal I)$, we have that $[A]_0 \in H_I$ for all $A \in \mathcal B$. Because $H_I$ is hereditary and $\theta^F_{i+1}([A]_i) = [A]_{i+1}$, $\theta^F_2([A]_1) = [A]_2$, for all $A \in \mathcal B$ and $i \geq 0$ we have that $[A]_i \in H_I$ and hence $p_{[A]_i} \in I$. Any $A \in \mathcal B^F$ is formed by some finite disjoint unions of these sets, so $p_A \in I$ as well, and we are done.
	\end{proof} 
	
	\end{proof}

	\section{Generalized Boolean Dynamical Systems as Generalized Labelled Spaces} \label{section:labelled}
	It has been established that generalized Boolean dynamical system algebras generalize the class of algebras derived from weakly left-resolving normal labelled spaces (see Example~\ref{example:labelledspace}). In this section, we will expand the class of labelled spaces by defining \textit{generalized labelled spaces} and show that the class of weakly left-resolving normal generalized labelled spaces algebras is equivalent to the class of generalized Boolean dynamical system algebras. This will act as a Stone duality and offer a graphical interpretation of a generalized Boolean dynamical system.

	\begin{definition}
	For a directed graph $\mathcal E$, an alphabet $\mathcal L$, and a labelling on the edges $\mathcal E^1 \rightarrow \mathcal L$, the pair $(\mathcal E, \mathcal L)$ is called a \textit{directed graph}. Note that we often abuse notation and refer to the labelling $\mathcal E^1 \rightarrow \mathcal L$ as $\mathcal L$ as well. For a subset $A \subseteq \mathcal E^0$ and $a \in \mathcal L$, we define the range operator $r(A, a) = \{v \in \mathcal E^0 \colon \exists e \in \mathcal E^1 \text{ such that } s(e) \in A, r(e) = v, \text{and } \mathcal L(e) = a\}$
	
	Let $\mathcal B$ be a set of subsets of $\mathcal E^0$ that is closed under finite unions and intersections and contains the emptyset. We say that $\mathcal B$ is an \textit{accommodating family} of $(\mathcal E, \mathcal L)$ if for all $A \in \mathcal B$ and $a \in \mathcal L$, we have that $r(A, a) \in \mathcal B$. A triplet $(\mathcal E, \mathcal L, \mathcal B)$ where $(\mathcal E, \mathcal L)$ is a directed graph and $\mathcal B$ is an accommodating family is called a \textit{labelled space}.

	If $\mathcal B$ is a closed under relative complements (i.e. $\mathcal B$ forms a generalized Boolean algebra), then the labelled space is called \textit{normal}. If for all $a \in \mathcal L$, the range operator $r(\cdot, a): \mathcal B \rightarrow \mathcal B$ preserves intersections, then the labelled space is called \textit{weakly left-resolving}. Note that if $(\mathcal E, \mathcal L, \mathcal B)$ is weakly left-resolving and normal, then $r(\cdot, a): \mathcal B \rightarrow \mathcal B$ is a morphism of generalized Boolean algebras.
	\end{definition}

	\begin{remark}
	Importantly, in our definition of an accommodating family, and contrary to other definitions \cite{bates2007c, Boava2021LeavittPA}, we do not require that the $r(a) = r(\mathcal E^0, a) \in \mathcal B$. This is similar to the jump made in \cite{CARLSEN2020124037} when moving from Boolean dynamical systems with compact range to generalized Boolean dynamical systems.
	\end{remark}

	\begin{definition}
	For a weakly left-resolving normal labelled space $(\mathcal E, \mathcal L, \mathcal B)$, define the ideal \[\mathcal F_a = \{A \in \mathcal B \colon A \subseteq r(B, a) \text{ for some } B \in \mathcal B\}\] If $\mathcal I$ is a set of ideals (treating $\mathcal B$ as a poset) indexed by $a \in \mathcal L$ such that $\mathcal F_a \subseteq \mathcal I_a$, we call the quadruplet $(\mathcal E, \mathcal L, \mathcal B, \mathcal I)$ a generalized weakly left-resolving normal labelled space.
	\end{definition}

	For the rest of this section, we will drop the notation ``weakly left-resolving normal'' and simply refer to our objects as generalized labelled spaces and labelled spaces.

	Fix a generalized labelled space $(\mathcal E, \mathcal L, \mathcal B, \mathcal I)$. We now define an associated $R$-algebra $L_R(\mathcal E, \mathcal L, \mathcal B, \mathcal I)$.

	\begin{definition}
	$L_R(\mathcal E, \mathcal L, \mathcal B, \mathcal I)$ is the universal associative $R$-algebra generated by the set $\{p_A\}_{A \in \mathcal B} \cup \{s_{a, A}, s_{a, A}^{\ast}\}_{a \in \mathcal L, A \in \mathcal I_a}$ and the following relations:
	\begin{enumerate}
        \item $p_{A \cap B} = p_A p_B$, $p_{A \cup B} = p_A + p_B - p_{A \cap B}$, and $p_{\emptyset} = 0$ for all $A, B \in \mathcal B$
        \item $p_B s_{a, A} = s_{a, A} p_{r(B, a)}$ and $s_{a, A}^{\ast} p_B = p_{r(B, a)} s_{a, A}^{\ast}$ for all $B \in \mathcal B$, $a \in \mathcal L$, and $A \in \mathcal I_a$
        \item $s^{\ast}_{a, A} s_{a', A'} = \delta_{a, a'} p_{A \cap A'}$ for all $a, a' \in \mathcal L$ and $A \in \mathcal I_a$ and $A' \in \mathcal I_{a'}$
        \item $s_{a, A} p_B = s_{a, A \cap B}$ and $p_B s_{a, A}^{\ast} = p_B s_{a, A \cap B}$ for all $a \in \mathcal L$, $B \in \mathcal B$, and $A \in \mathcal I_a$
        \item $p_A = \sum_{a \in \Delta_A} s_{a, r(A, a)} s_{a, r(A, a)}^{\ast}$ for all $A \in \mathcal B_{\text{reg}}$
    \end{enumerate}

	\end{definition}
	
	We now list two theorems that relate generalized labelled spaces to previous work. Their proofs are both easy, but tedious, checks on relations.

	\begin{theorem}
	If $r(a) \in \mathcal B$ for all $a \in \mathcal L$, then, defining $\mathcal I_a = \{A \in \mathcal B \colon A \subseteq r(a)\}$ for all $a \in \mathcal L$, we have that \[L_R(\mathcal E, \mathcal L, \mathcal B) \cong L_R(\mathcal E, \mathcal L, \mathcal B, \mathcal I)\] where $L_R(\mathcal E, \mathcal L, \mathcal B)$ is the labelled Leavitt path algebra defined in \cite{Boava2021LeavittPA}.
	\end{theorem}

	This theorem shows that when $B$ is accommodating in the sense that it contains $r(a)$ for all $a \in \mathcal L$ as required by \cite{Boava2021LeavittPA}, then, defining $\mathcal I$ appropriately, we can construct a generalized labelled space algebra that is isomorphic to the labelled Leavitt path algebra. An easy corollary is that generalized labelled space algebras are a generalization of labelled Leavitt path algebras. This is the same idea as \cite[Example~4.1]{CARLSEN2020124037}. 

	\begin{theorem}
	Define the morphisms of generalized Boolean algebras $\theta_a \coloneqq r(\cdot, a): \mathcal B \rightarrow \mathcal B$ for all $a \in \mathcal L$. Then $(\mathcal B, \mathcal L, \theta, \mathcal I)$ is a generalized Boolean dynamical system and \[L_R(\mathcal E, \mathcal L, \mathcal B, \mathcal I) \cong L_R(\mathcal B, \mathcal L, \theta, \mathcal I)\]
	\end{theorem}

	This theorem shows that all labelled space algebras can be written as generalized Boolean dynamical system algebras. To show equivalence, it remains to show the converse. To do this, we will employ the Stone dual, which we now define. To minimize confusion of notation, we will refer to elements of $\mathcal B$ with the lower-case letters and reserve uppercase letters for subsets of $\mathcal B$, contrary to the other portions of the paper.

	\begin{definition}
	A subset $F \subseteq \mathcal B$ of a generalized Boolean algebra is called a \textit{filter} if:
	\begin{enumerate}
		\item $F \neq \emptyset$ and $F \neq \mathcal B$
		\item For all $x \in F$ and $y \in \mathcal B$ where $x \leq y$, we have that $y \in F$
		\item For all $x, y \in F$, we have that $x \wedge y \in F$
	\end{enumerate}
	For $x \in \mathcal B$, denote $V_x = \{F \subseteq \mathcal B \text{ is a filter } \colon x \in F\}$ to be the set of filters containing $x$. Note that for all $x, y \in \mathcal B$ we have that $V_x \cup V_y = V_{x \vee y}$, $V_x \cup V_y = V_{x \wedge y}$, and $V_x \setminus V_y = V_{x \setminus V_y}$. Namely, the set $\{V_x \colon x \in \mathcal B\}$ forms a generalized Boolean algebra on the filters of $\mathcal B$.
	\end{definition}

    From now on, fix a generalized Boolean dynamical system $(\mathcal B, \mathcal L, \theta, \mathcal I)$. Denote the set of all filters on $\mathcal B$ as $X(\mathcal B)$ and define $\mathcal E_g^0 \coloneqq X(\mathcal B)$. For $a \in \mathcal L$, extend the domain of $\theta_a$ to subsets of $\mathcal B$ in the obvious way. For $F, F' \in X(\mathcal B)$ and $a \in \mathcal L$, add an edge $(F, F')$ with label $a$ to $\mathcal E^1$ if $\theta_a(F) \subseteq F'$. This forms a labeled graph $(\mathcal E_g, \mathcal L_g)$. As mentioned previously, the set $\mathcal B_g \coloneqq \{V_x \colon x \in \mathcal B\}$ forms a generalized Boolean algebra on $\mathcal E_g^0$. To prove that $\mathcal B_g$ is an accommodating family, we must show that $r(V_x, a) \in \mathcal B$ for all $x \in \mathcal B$. From the following computation, this will be obvious.

	\begin{lemma} \label{lemma:boolrangecomp} $r(V_x, a) = V_{\theta_a(x)}$ for $x \in \mathcal B$.
    \end{lemma}
    \begin{proof}
    We first show that $r(V_x, a) \subseteq V_{\theta_a(x)}$. For $F \in V_x$, any edge emitted from $F$ labelled by $a$ must be to a filter $F'$ such that $\theta_a(F) \subseteq F'$. Because $\theta_a(x) \in \theta_a(F)$, we have that $\theta_a(x) \in F'$ and hence $F' \in V_{\theta_a(x)}$.

    We now show that $r(V_x, a) \supseteq V_{\theta_a(x)}$. Let $F$ be a filter that contains $\theta_a(x)$. Consider the set $\theta_a^{-1}(F) = \{x \in \mathcal B \colon \theta_a(x) \in F\}$. It's well known that $\theta_a^{-1}(F) \in V_x$ is also a filter. Furthermore, we have that $\theta_a(\theta_a^{-1}(F)) \subseteq F$ so there is an edge $(\theta_a^{-1}(F), F)$ labelled with $a$. Hence, $F \in r(V_x, a)$ and we are done.
    \end{proof}

	Using Lemma~\ref{lemma:boolrangecomp}, we see that $r(V_x, a) \cap r(V_y, a) = V_{\theta_a(x)} \cap V_{\theta_a(y)} = V_{\theta_a(x) \wedge \theta_a(y)} = V_{\theta_a(x \wedge y)} = r(V_{x \wedge y}, a) = r(V_x \cap V_y, a)$ which shows that the space $(\mathcal E, \mathcal L, \mathcal B)$ is weakly left-resolving. In addition, the space is obviously normal.

	Using the isomorphism $\mathcal B \cong \mathcal B_g$, we form ideals $(\mathcal I_g)_a$ that correspond to each $\mathcal I_a$. These satisfy our required properties, so we form a generalized labelled space $(\mathcal E_g, \mathcal L_g, \mathcal B_g, \mathcal I_g)$.

	\begin{theorem}
		$L_R(\mathcal B, \mathcal L, \theta, \mathcal I) \cong L_R(\mathcal E_g, \mathcal L_g, \mathcal B_g, \mathcal I_g)$ through the map \[p_x \mapsto p_{V_x} \text{ for all } x \in \mathcal B \]\[s_{a, y} \mapsto s_{a, V_y} \text{ for all } a \in \mathcal L, y \in \mathcal I_a\]
	\end{theorem}

	\begin{proof}
		The proof just uses Lemma~\ref{lemma:boolrangecomp} and is the same standard symbolic manipulations.
	\end{proof}

	\begin{corollary}
	The class of generalized Boolean dynamical system algebras and the class of generalized labelled space algebras are the same.
	\end{corollary}

	\bibliographystyle{abbrv}
	\bibliography{DynamicSystem}
	
\end{document}